\documentclass[
leqno, 
12pt]{amsart}
\usepackage{amssymb,amsmath, amsthm,latexsym, mathrsfs}

\usepackage{a4wide}

\usepackage{hyperref}

\theoremstyle{plain}

\newtheorem{lemma}{Lemma}[section]
\newtheorem{theorem}[lemma]{Theorem}
\newtheorem{proposition}[lemma]{Proposition}

\newtheorem{remark}[lemma]{Remark}

\newtheorem{definition}[lemma]{Definition}

\newtheorem{example}[lemma]{Example}

\parskip=\bigskipamount

\font\rm=cmr12

\def\Z{\mathbb Z}
\def\R{\mathbb R}
\def\d{\delta}
\def\r{\rtimes}
\def\t{\times}
\def\o{\otimes}
\def\e{\epsilon}

\def\a{\alpha}
\def\b{\beta}

\def\s{\sigma}

\title[Bounds on unitary  duals]
{Some bounds on unitary  duals of classical groups - non-archimeden case}

\mark{Marko Tadi\'c}

\author{Marko Tadi\'c}

\address{Department of Mathematics, University of Zagreb
\\
Bijeni\v{c}ka 30, 10000 Zagreb,
 Croatia\\
Email: \tt tadic{\char'100}math.hr}


\thanks{This work has been  supported  by Croatian Science Foundation under the
project 9364.}
\date{\today}

\begin{document}

{\sl Preliminary version}

\begin{abstract} 
In the first part of the paper we give some bounds for domains where the unitarizabile subquotients can show up in the parabolically induced representations  of  classical $p$-adic groups. Roughly, it can show up only if the 
central character of the inducing irreducible cuspidal representation is dominated in an appropriate way by the  
square root of the modular character of minimal parabolic subgroup.
For representations supported by fixed parabolic subgroup, a  more precise bound is given. There are also bounds for specific Bernstein components. 
A number of these upper bounds are best possible.

The second part of the paper addresses a question how far is the trivial representation from the rest of the unramified automorphic dual. 
By a result of L. Clozel, trivial representation is isolated in the automorphic dual of a split rank one semisimple group over a completion of a global field, but it is very seldom isolated in the unitary dual (it can happen only in the archimedes cases). Further, the level of isolation in the case of $SL(2)$  is  important  for the number theory. For the higher rank groups, the trivial representation is always isolated in the unitary dual by an old result of D. Kazhdan. Still, we may ask  if the level of isolation is higher in the case of the automorphic duals. We show that the answer is negative to this question for symplectic $p$-adic groups.
\end{abstract}

\maketitle

\setcounter{tocdepth}{1}


\section{Introduction}\label{intro}

Bounds on various parts of unitary duals of  reductive groups over local fields can be very important, in particular in the number theory. Let $\rho$ be an irreducible cuspidal representation of a Levi factor $M$ of a parabolic subgroup $P$ of a connected reductive group over a non-archimedean local field $F$. In \cite{T-top-dual} we have proved that the set of all unramified characters $\chi$ of $M$ such that Ind$_P^G(\chi)$ contains an irreducible  unitarizable subquotient is a compact subset of the set of all unramified characters  of $M$\footnote{This result is very easy to prove using Bernstein center (see \cite{T88}). Further, one can drop the condition of cuspidality.} (observe that this fact does not hold for archimedean fields, already for $SL(2,\mathbb R)$). In other words, this implies that the set of such characters where irreducible unitarizable subquotients can show up is bounded (for each fixed $M$ and $\rho$). Two questions were left unanswered there. The first question is  to find good bounds for the region where the unitarizability can show up  in a fixed Bernstein component\footnote{Bernstein component of a non-unitary dual determined by $\rho$ is the set of all equivalence classes of irreducible subquotients of representations Ind$_P^G(\chi\rho)$ when $\chi$ runs over the set of all unramified characters of $M$.}, i.e. how far from the unitary axis we can have unitary subquotients (for fixed $\rho$).
 The second question is if
the set of all these regions where unitarizability can show up is bounded  
   when one fixes $M$, and let $\rho$ to run over all the irreducible cuspidal representations of $M$ (i.e., if   there exists a set of bounds which is bounded), and if it exists, to find as precise common bound as possible.

 A result in the direction of the first question was obtained by L. Clozel and E. Ullmo in \cite{ClU04} for $G=Sp(2n,F)$, $M=A$ the maximal torus of $G$ and $\rho=\mathbf 1_{A}$, the trivial representation of $G$, and for the subset of automorphic (unitarizable) representations in this (unramified) Bernstein component. We shall briefly recall of their result.
 First we introduce some notation.

On $\mathbb R^q$ we shall consider two orderings. Let $x=(x_1,\dots,x_q) , y=(y_1,\dots,y_q) \in \mathbb R^q$. Then we write
$$
\textstyle
 x\leq_w y\iff \sum_{j=1}^i x_j\leq    \sum_{j=1}^i y_j,\ \ \ \ \forall i\in\{1,\dots,q\},
$$
$$
x\leq_s y\iff x_i\leq y_i,\ \ \ \ \forall i\in\{1,\dots,q\}.
$$

Denote by $\mathcal O_F$ the maximal compact subring of $F$, and by $q_F$ the cardinality of its residual field. Fix an element  $\varpi_F$ which generates the maximal ideal in $\mathcal O_F$.  Denote by $|\ |_F$ the normalized absolute value on $F$ (it is determined by the condition $|\varpi_F^{-1}|_F=q_F$).
 For $i=1,\dots,q$, denote by $e_i=(1,\dots,1,\varpi_F^{-1},1,\dots,1)$, where $\varpi_F^{-1}$ is placed  at the $i$-th place.

We consider a maximal torus $A$ in $Sp(2q,F)$ consisting of the diagonal matrices in the group (see the second  section for more details regarding notation). Using an isomorphism $(x_1,\dots,x_q)\mapsto (x_1,\dots,x_q,x_q^{-1},\dots, x_q^{-1})$, we identify  $(F^\times)^q$ with the maximal torus.  Denote by $P_{min}$ a minimal parabolic subgroup consisting of the diagonal matrices in $Sp(2q,F)$.  

Fix an irreducible  representation $\pi$ of $Sp(2q,F)$. Then we can find a standard parabolic subgroup $P=MN$ of $G$ such that the Levi factor $M$ contains $P_{min}$, an irreducible unitarizable cuspidal representation $\rho$ of $M$ and a positive valued (unramified) character $\chi$ of $M$ such that $\pi$ is a subquotient of Ind$_{P}^{Sp(2q,F)}(\chi\rho)$, and that holds
$$
\log_{q_F}(\chi(e_1))\geq \dots \geq \log_{q_F}(\chi(e_q)_F)\geq 0.
$$
Then the element $(\log_{q_F}(\chi(e_1)), \dots , \log_{q_F}(\chi(e_q)_F))\in  (\mathbb R_{\geq 0})^q$ is uniquely determined by $\pi$, and it will be denoted by
$$
||\pi||.
$$
Recall that for the trivial representation $\mathbf 1_{Sp(2q,F)}$ we have
$$
||\mathbf 1_{Sp(2q,F)}||=(q,q-1,\dots,2,1).
$$

Now we shall recall of Theorem 6.3 from \cite{ClU04}  (we shall use logarithmic interpretation of it):

\begin{theorem}
\label{Clozel-Ullmo}
 {\rm  (L. Clozel, E. Ullmo)} Let $k$ be a number field and $v$ a place of $k$. Denote by $F$ the completion of $k$ at $v$. Let $\pi$ be a non-trivial irreducible unramified representation of $Sp(2q,F)$ which shows up as a tensor factor of an irreducible representation of the adelic group $Sp(2q,\mathbb A_k)$ which is in the support of the representation $Sp(2q,\mathbb A_k)$ in the space of the square integrable automorphic forms $L^2(Sp(2q,k)\backslash Sp(2q,\mathbb A_k)$). Then
$$
||\pi||\leq_w (q-1+\theta,q-2+\theta,\dots,1+\theta,\theta),
$$
where $\theta$ is  the Ramanujan constant for $SL(2,k)$\footnote{See  \cite{ClU04} for more details.}.
\end{theorem}

The proof of this theorem 
is based on a result proved by  L. Clozel and E. Ullmo in \cite{ClU04}, claiming that restriction preserves automorphicity. Later in  \cite{Cl07},
using \cite{Oh02},  L. Clozel avoids use of the Ramanujan constant if $q\geq2$, showing that     for $q\geq2$,  $\theta$ can be taken to be 0 in the above estimate of $||\pi||$. 

In particular, if we drop in the   above theorem the assumption that $\pi$ is non-trivial, we get obviously the estimate $(q,q-1,\dots,2,1)$ for (unitarizable) automorphic unramified  irreducible representations, i.e. 
$$
||\pi||\leq_w (q,q-1,\dots,2,1).
$$
In other words, for automorphic representation in the unramified component the upper bound is $(q,q-1,\dots,2,1)$. 
In this paper we prove that for the above estimate we can avoid assumption of automorphicity, and moreover, that this estimate holds for any irreducible unitarizable representation (not only unramified). 
Actually, we prove a slightly stronger result then this, since the estimate is for $\leq_s$:

\begin{theorem}
\label{intro-general}
  Let $F$ be a local non-archimedean field of characteristic 0  and let $\pi$ be an irreducible unitarizable   representation of $G:=Sp(2q,F)$. Then
$$
||\pi||\leq_s ||\mathbf1_G||.
$$
The equality  holds if and only $\pi$ is the trivial or the Steinberg  representation.
\end{theorem}

The proof of the above theorem is a relatively easy consequence of \cite{T-CJM} and a recent work of J. Arthur and C. M\oe glin, which we use to get estimate for the  cuspidal reducibility\footnote{It s possible that the above simple estimate is already known, but we do not know for it.}. 
The above theorem holds also for other classical groups (with slight modification; see the third section). 
It would be interesting to know if such an estimate holds  more  generally (for $\leq_w$).

For particular $M$, we can often get much better estimates. They are in the following

\begin{theorem}
\label{intro-fixed-parabolic} Let char$(F)=0$ and let
 $P$ be 
a parabolic subgroup of $Sp(2n,F)$ whose Levi factor is isomorphic to
$$
GL(p_1,F)^{n_1}\t\dots GL(p_k,F)^{n_k}\t Sp(2q,F),
$$ 
where $n_i\ne n_j$ for $i\ne j$.
Let $\pi$ be an irreducible unitarizable subquotient of 
$$
\text{Ind}_P^{Sp(2n,F)}(|\det|_F^{e_1}\rho_1\t\dots\t|\det|_F^{e_k}\rho_k\r\s),
$$
 where $\rho_i$ are irreducible unitarizable cuspidal representations of general linear groups, $e_i\in\R$ and $\sigma$ is an irreducible cuspidal representations of $Sp(2q,F)$. 
 Denote by $e_1',e_2', \dots,e_{n_{i_0}}'$ the subsequence of the sequence $e_1,e_2, \dots,e_k$ consisting  of all $e_i$ such that $\rho_i$ is a representation of $GL(n_{i_0},F)$.
 After renumeration, we can assume $|e_1'|\geq |e_2'|\geq \dots\geq |e_{n_{i_0}}'|$.
Then  
$$
\textstyle
(|e_1'|,|e_2'|, \dots,|e_{n_{i_0}}'|)
 \leq_s (r,r-1,\dots,c+1,c)
$$
for appropriate $r$, where
$$
\textstyle
c=\max\big\{t\in 1/2\Z; t\leq \sqrt{\frac{2q}{n_{i_0}+1}}\big\}.
$$

\end{theorem}

For other classical groups the upper bound $c$ is very similar (see the third section).

Consider the example where $k=1$, $p_1=2$, $n_1=5$ and $q=6$ in the above theorem. The above theorem gives the bound $(\frac{13}2,\frac{11}2,\frac92,\frac72,\frac52)$. In other words, 
\begin{equation}
\label{example-intro}
\textstyle
||\pi||\leq_s (\frac{13}2,\frac{13}2,\frac{11}2,\frac{11}2,\frac92,\frac{9}2,\frac72,\frac{7}2,\frac52,\frac{5}2,0,0,0,0,0,0).
\end{equation}
This is much sharper estimate then the one given by Theorem \ref{intro-general},  which  is $(16,15,\dots,2,1)$.

If we want upper bound for specific Bernstein component, we need to have a data parameterizing  the cuspidal representation $\s$. We shall use Jordan blocks (see the third section). Fix an irreducible  square integrable representation $\s$ of some $Sp(2q,F)$ and fix a self dual irreducible cuspidal representation $\rho$ of a general linear group. For $k\in \Z_{\geq 1}$, the square integrable representation attached by Bernstein-Zelevinsky to the segment $[ |\det|_F^{-(k-1)/2}\rho,|\det|_F^{(k-1)/2}\rho]$ will be denoted by $\d(\rho,k)$ (see the third section).
Then for one parity of $k\in\Z_{\geq 1}$, the representation Ind$(\d(\rho,k)\o\mathbf \s)$ is always irreducible\footnote{It is even if Ind$( |\det|_F^{1/2}\rho\o \mathbf1_{Sp(0,F)})$ is irreducible. Otherwise it is odd.}. For the other parity we have reducibility, with finitely many exceptions. Denote by $Jord(\s)$ the set of all such representations which are exceptions, when $\rho$ runs over all the self dual irreducible cuspidal representations of  general linear groups\footnote{C. M\oe gain has shown that 
using the local Langlands correspondence for
general linear groups, $Jord_\rho(\s)$ transfers to  the admissible homomorphism of the Weyl-Deligne group attached  by J. Arthur to $\s$ in \cite{A} (see \cite{T13} for a little bit more precise explanation, but still avoiding  too much technical details).}. For a self dual irreducible cuspidal representation $\rho$ of a general linear group, denote by $Jord_\rho(\s)$ the set of all $k$ such that there exist $\d(\rho,k)\in Jord(\s)$.

The following theorem gives upper bounds for individual Bernstein components.

\begin{theorem} 
\label{intro-Bernstein-component}
Let char$(F)=0$. Fix an  irreducible cuspidal representation $\sigma$ of $Sp(2q,F)$ and let $\rho_1,\dots,\rho_k$ be  irreducible unitarizable cuspidal representations of general linear groups $GL(p_1,F), \dots , GL(p_k,F)$ respectively, such that $\rho_i\not\cong|\det|_F^\b\rho_j$ for any $i\ne j$ and any $\a\in \mathbb C$. Let $n_1,n_2,\dots,n_k$ be positive integers and let $\b_{i,j}, 1\leq i\leq k, 1\leq j\leq n_i$ be a set of complex numbers such that
the representation $$
|\det|_F^{\b_{1,1}}\rho_1\t\dots \t|\det|_F^{\b_{1,n_1}}\rho_1\t\dots
$$
 $$
\dots\t |\det|_F^{\b_{k,1}}\rho_1\t\dots \t|\det|_F^{\b_{k,n_k}}\rho_1\r\s
$$
contains a unitarizable subquotient.
 Fix some index ${i}$. After renumeration we can assume that for real parts of complex exponents hold $|\Re(\b_{i,1})| \geq  \dots\geq |\Re(\b_{i,n_i})|$. Denote by $X_i$ the set of all unramified  characters $\chi$ of $GL(p_i,F)$ such that $\chi\rho_i$ is a self dual representation. Then $X_i$ is a finite set. If $X_i\ne\emptyset$, denote
 $$
\textstyle
c_i=\frac{1+\max\{\text{card}(Jord_{\chi\rho_i}(\s));\chi\in X_i\}}2.
$$
Then
$$
\textstyle
(|\Re(\b_{i,1})|,  \dots, |\Re(\b_{i,n_i})|)
 \leq_s (c_i+n_i-1,\dots,
 c_i+1,c_i)
$$
 if $X_i\ne\emptyset$, and 
$$
\textstyle
(|\Re(\b_{i,1})|,  \dots, |\Re(\b_{i,n_i})|)
 \leq_s (\frac {n_i}2,\frac{n_i-1}2,\dots,1,\frac12)
$$
if $X_i=\emptyset$.
\end{theorem}

 The second topic of this paper is also devoted to bounding some parts of  the unitary dual, but in a different way. We shall explain this below. It is again related to the result of L. Clozel and E. Ullmo which we have mentioned above.

Let us recall that D. Kazhdan introduced in \cite{K67} property (T) for a locally compact group $G$. This property means that the trivial representation is isolated in the unitary dual $\hat G$ of $G$, i. e. in the set of all equivalence classes of the irreducible unitary representations of $G$, supplied with a natural topology.  He proved there that a simple group $G(F)$ over a  locally compact non-discrete  field $F$ of split rank different from one has property (T), and he obtained some very interesting  arithmetic consequences from that. Property (T) has shown to be related to a number of interesting facts. There is a vast literature on this, which we shall not discuss here. We shall mention only one result  in that direction  related to the exception of rank one in the Kazhdan result. 

L. Clozel proved in \cite{Cl03} $\tau$-conjecture, which is a weaker condition then property (T), but holds for any rank. This weaker property tells that the  trivial representation is isolated in the automorphic dual of $G(F)$. More precisely, if $F$ is a completion of a number field $k$ at a place $v$, then the automorphic dual of $G(k)$ at $v$ is the support of the representation  $G(k_v)$ on the space $L^2(G(k)\backslash G(\mathbb A_k))$ of the square integrable automorphic forms.

Now that one knows that the trivial representation is isolated in the automorphic dual, further important question may be: how far is it from the other irreducible automorphic representations (in other words, what is the quantitative  level of isolation of the trivial representation). Let us recall that   such type of questions in rank one case
are important in the  number theory. One may ask such question also in the higher ranks. Recall that the trivial representation is not isolated in the unitary dual in rank one case, except in several archimedean completions. In the higher ranks, the trivial representation is isolated in both unitary and automorphic. Considering the rank one case, one may ask if  the level of isolation is higher in the automorphic dual then in the unitary dual in higher ranks.
Theorem \ref{Clozel-Ullmo} of L. Clozel and E. Ullmo for symplectic groups   gives an estimate in the automorphic case (for $q\geq 2$, taking $\theta=0$). Moreover, this estimate cannot be improved in this case (for $\theta=0$), since G.~Mui\'c has proved in  \cite{Mu-unit} that there exists an irreducible unitarizable  automorphic representation for which we have equality in their theorem (this representation is the Aubert-Schneider-Stuhler involution of an irreducible square integrable representation which is the "closest" to the Steinberg representation\footnote{G. Savin has also told us of such an example.}).

In this paper we prove the following result:

\begin{theorem} 
\label{intro-unramified}
 Let $F$ be a local non-archimedean field of characteristic different from two  and let $\pi$ be a non-trivial irreducible unramified representation of $Sp(2q,F)$, $q\geq 2$. Then
$$
||\pi||\leq_w (q-1,q-2,\dots,1,0)
$$
and
$$
\textstyle
||\pi|| <_s (q-1,q-2,\dots, 2,1,\frac12).
$$
\end{theorem}

The first inequality of the theorem tells that the trivial representation is not more isolated in the automorphic dual than it is in the unitary dual (here we consider the ranks at least two).

The proof of the above theorem is an elementary application of  the classification of the unramified  unitary duals in \cite{MuT} (this also reproves the result of L. Clozel and E. Ullmo, and adds a new estimate there).

 We are thankful to Goran Mui\'c for discussions during writing of this paper.
 
 The contents of this paper is as follows. After introduction, in the second section we introduce notation that shall be used in the paper. The third section is devoted to the bounds where  the unitarizability can show up. We recall of the unramified unitary dual of $p$-adic symplectic groups in the fourth section, while the fifth section brings bounds for the trivial representation from the rest of the unramified irreducible unitary representations.

\section{Notation}

We fix a local non-archimedean field $F$.
The normalized absolute value on $F$ is denoted by $|\ \ |_F$ and the
character $g\mapsto |\text{det}|_F$, $GL(n,F)\rightarrow \mathbb R^\times$ is
denoted by $\nu$. For each irreducible essentially square integrable
representation $\delta$ of $GL(n,F)$ there exist a unique $e(\delta)\in \mathbb R$ and a unique up to an
equivalence (unitarizable) irreducible square integrable representation
$\delta^u$ of $GL(n,F)$ such that 
$$
\delta\cong\nu^{e(\delta)}\delta^u.
$$

For  smooth
representations $\pi_i$  of $GL(n_i,F)$ for $i=1,2$, 
 by
$\pi_1\times\pi_2$ is denoted  the smooth representation of $GL(n_1+n_2,F)$
parabolically induced by 
$$
\pi_1\otimes\pi_2
$$
 from appropriate  maximal
parabolic subgroup, which is standard with respect to the subgroup of the
upper triangular matrices (parabolic induction that we consider is
normalized). 

Fix an irreducible square integrable representation $\tau$ of a general linear group and a positive integer $k$. Then the representations
$$
\nu^{(k-1)/2}\tau\t \nu^{(k-1)/2-1}\tau\t \dots\t \nu^{-(k-1)/2}\tau 
$$
has a unique irreducible quotient, which is denoted by
$$
u(\tau,k).
$$

For an irreducible square integrable representation $\tau$ of a general linear group there exists an irreducible unitarizable cuspidal representation $\rho$ of a general linear group and a positive integer $\ell$ such that $\tau$ is isomorphic to the unique irreducible subrepresentation of 
$$
\nu^{(\ell-1)/2}\rho\t \nu^{(\ell-1)/2-1}\rho\t \dots\t \nu^{-(\ell-1)/2}\rho .
$$
Then we write
$$
\tau=\d([\nu^{-(\ell-1)/2}\rho,\nu^{(\ell-1)/2}\rho]).
$$

For  an irreducible representation $\pi$ of $GL(n,F)$ there exist
irreducible cuspidal representations $\rho_1,\dots,\rho_k$ of general linear
groups such that $\pi$ is isomorphic to  a subquotient of $\rho_1\times\dots
\times\rho_k$. The multiset of equivalence classes $(\rho_1,\dots,\rho_k)$
is called the cuspidal support of $\pi$, and it is denoted by supp$(\pi)$.

While in the case of general linear groups we follow mainly notation of   \cite{Ze}, in the case of
 the
classical $p$-adic groups, we follow the notation of \cite{T-Str}. The $n\times n$ matrix having 1's on the second
diagonal and
 all other entries  $0$ is  denoted by $J_n.$  The identity
matrix is denoted by   $I_n$.   For a  $2n \times 2n$  matrix $S$,  denote 
$$
\aligned
^{\times}S = \left[
\begin{matrix}
 0 & -J_n\\
J_n & 0
\end{matrix}
\right] \;
^tS  \left[
\begin{matrix}
0 & J_n\\
-J_n & 0
\end{matrix}
\right].
\endaligned
$$
Then ${Sp}(n,F)$ is the group of all  $2n \times 2n$   matrices over $F$ which
satisfy  $^\times S \, S = I_{2n}$ (we take  ${Sp}(0,F)$ to be the trivial group).

By $O(m,F)$ is denoted the group of all $m
\times
m$ matrices $X$  with entries in F which satisfy
$
^\tau X\, X =  I_{m}.
$
Denote $SO(m,F)=O(m,F)\cap SL(m,F)$.

In the case of groups that we consider in this paper, we always fix the minimal  parabolic subgroup   
consisting of all upper triangular matrices in the 
group (denoted by $P_{\text{min}}$).

In the sequel, we denote by $S_n$  either the group  ${Sp}(n,F)$ or
${SO}(2n+1,F)$. Parabolic subgroups which contain the minimal parabolic
subgroup which we have fixed will be called standard parabolic subgroups.

Most of the results of the section 3. hold also for other classical groups considered in \cite{Moe-T}, and also for unitary groups (also considered in \cite{Moe-T}). When some statement of this paper  is specific for symplectic or split odd-orthogonal groups, it will be specified in the statement.

In this case of unitary groups, a  separable quadratic extension $F'$ of $F$ is fixed. We denote by $\theta$ the non-trivial element of the Galois group of $F'$ over $F$. 
Whenever in the non-unitary case appears  a representations of $GL(n,F)$, we need to replace it with a representations of $GL(n,F')$, and 
the contragredient representations $\tilde\pi$ of a representation $\pi$ of $GL(n,F)$, with the representations $g\mapsto \pi(\theta(g))$ of $GL(n,F)$\footnote{We can handle the case of unitary groups uniformly, as we did in \cite{Moe-T}.}.

 The Jacquet module of a
representation $\pi$ of $S_n$ for the standard maximal parabolic subgroup
whose Levi factor is a direct product of $GL(k,F)$ and a classical group $S_{n-k}$, is
denoted by
$$
s_{(k)}(\pi).
$$

Each irreducible representation $\tau$  of some classical group $S_\ell$. is a subquotient  of  a representation of the form
$$
\rho_1\times\dots\dots \times
\rho_k\rtimes\sigma,
$$
 where $\rho_1,\dots 
,\rho_k$ are  irreducible cuspidal representations    of general linear groups
and $\sigma$ is an irreducible cuspidal representation  
 of some $S_m$.
The representation $\sigma$ is called the partial cuspidal support of $\tau$
and it is denoted by
$$
\tau_{cusp}.
$$
  If $\rho_1\times\dots\dots \times
\rho_k$ is a representation of $GL(p,F)$, then the Jacquet module
$s_{(p)}(\tau)$ will be denoted by 
$$
s_{GL}(\tau).
$$
An irreducible cuspidal representation $\rho$ of a general linear group is
called a factor of $\tau$ if there exists an irreducible subquotient
$\pi\otimes\tau_{cusp}$ of
$s_{GL}(\tau)$ such that $\rho$ is in the cuspidal support of $\pi$. Then the
set of all factors of
$\tau$ is contained in
$$
\{\rho_1,\tilde\rho_1,\rho_2,\tilde\rho_2,\dots, 
\rho_k,\tilde\rho_k \}
$$
(recall that our $\tau$ is a subquotient of $\rho_1\times\dots\dots \times
\rho_k\rtimes\sigma $). Further, for each $1\le i\leq k$, at least
one representation from
$\{\rho_i,\tilde\rho_i\}$ is a factor of $\tau$.

Let $\rho$ and $\sigma$ be unitarizable irreducible cuspidal representations
of a general linear group and of $S_n$ respectively. Then if
$\nu^\alpha\rho\r \s$ reduces for some $\alpha\in \mathbb R$, then
$\rho\cong\tilde\rho$. Further, if $\rho\cong 
\tilde\rho$, then we have always reduction for
unique $\alpha\geq 0$ (\cite{Si}). This reducibility point will be denoted by
$$
\alpha_{\rho,\sigma}.
$$
A very non-trivial fact which  follows from the recent work of J. Arthur and C. M\oe glin is that always
$
\alpha_{\rho,\sigma}\in (\frac12)\mathbb Z
$
if char$(F)=0$.

\section{Bounding unitarizability}

The following proposition  from \cite{T-CJM} (Proposition 2.2 there) will be used several times in this paper to get some bounds where the unitarizability can show up in the parabolically induced representations (it was also used in \cite{T-CJM} for similar purpose).

\begin{proposition}
\label{CJM}
Let $\pi$ be an irreducible  representation of a classical group $S_q$. 

\noindent  (i) Let
$X$ be a set of irreducible cuspidal representations of general linear groups
which satisfies
\begin{enumerate}
\item $\nu^{\pm1}\rho\not\in \tilde X$, for any $\rho\in X$.

\item $X\cap \tilde X=\emptyset$.

\item There is no element in $X\cup \tilde X$ which is a factor of  $\pi$.

\item $\rho\r\pi_{cusp}$ is irreducible for any $\rho\in X$.

\item $\rho \t\rho'$ and $\tilde\rho \t\rho'$ are irreducible for any
$\rho\in X$ and any factor
$\rho'$ of $\pi$.

\end{enumerate}
 Suppose that $\theta$ is an irreducible representation of a general linear
group whose cuspidal support is contained in $X$. Then
$$
\theta\r\pi
$$
is irreducible.

\noindent
(ii) Suppose that we can find sets $X$ and $Y$ of (equivalence classes of)
irreducible cuspidal representations  of general linear groups such that
$X\cup
\tilde X \cup Y\cup \tilde Y$ contains all the factors of $\pi$, $X\cap
(Y\cup \tilde Y)=\emptyset$, and that  hold conditions (1), (2) and (4)
from (i). Further suppose that $\rho \t\rho'$ and $\tilde\rho \t\rho'$ are
irreducible for all $\rho\in X\cup \tilde X$ and $\rho'\in Y$
(i.e. that holds condition (5) from (i) for all  
$\rho\in X\cup
\tilde X$ and
$\rho'$ in
$Y$).

Then there exists an irreducible representation $\theta$ of a general linear
whose  cuspidal  support  is
contained in
$X$ (i.e. each representation of the support),
and there exists
an irreducible representation $\pi'$ of a classical group whose  all factors
are contained in
$Y\cup\tilde Y$,  such that
$$
\pi\cong
\theta
\rtimes
\pi'.
$$
 The partial cuspidal support  of $\pi'$ is $\pi_{cusp}$. Further, $\pi$
determines $\theta$ and $\pi'$ as above up to  equivalence.

If $X$ is a subset of the set of all  the factors of $\pi$, then  each
representation from $X$ shows up  in the
 cuspidal support of $\theta$.

\end{proposition}

We shall use it also in this paper to get some additional bounds.

\begin{proposition}
\label{bounds-general}
 Let $\pi$ be an irreducible unitarizable
representation of a classical group $S_q$. 
Let $\rho$ be a factor of $\pi$.
Suppose that $\rho_1,\dots,\rho_n$ are all the  factors $\tau$ of $\pi$
such that
$\tau^u\cong\rho^u$.

\begin{enumerate}

\item
Let
$
\rho^u\not\cong\widetilde{\rho^u}.
$
 Renumerate $\rho_1,\dots,\rho_n$, $n\geq1$, in a way that
$
|e(\rho_1)|\leq|e(\rho_2)|\leq\dots\leq|e(\rho_n)|$. Then
$$
\textstyle
|e(\rho_i)|\leq \frac i2, \quad 1\leq i \leq n
$$

\item
Suppose
$
\rho^u\cong\widetilde{\rho^u}.
$
 Write the set of all
$
|e(\rho_i)|>
\alpha_{\rho^u,\pi_{cusp}} $, 
 $1\leq i\leq n,
$
as 
$$
\{\alpha_1,\dots,\alpha_\ell\},
$$
 where $\ell \geq 0$ and
$\alpha_1<a_2<\dots <\alpha_\ell$. Then
$$
\alpha_i-\alpha_{i-1}\leq 1
\quad \text{for each} \quad
i=2,3,\dots,\ell,
$$
if $\ell\geq2$. Further

\

\medskip
\noindent
(i) If $\alpha_{\rho^u,\pi_{cusp}}=0$, 
 then
$$
\textstyle
\alpha_i\leq i-\frac12;\quad
i=1,\dots,\ell.
$$

\

\medskip
\noindent
(ii) If $\alpha_{\rho^u,\pi_{cusp}}\geq \frac12$, then 
there exists index
$i$ such that
$$
|e(\rho_i)|
\leq \alpha_{\rho^u,\pi_{cusp}}.
$$
Denote by
$$\alpha_{\rho^u,\pi_{cusp}}^{(\pi)}
=\max \{|e(\rho_i)|;  |e(\rho_i)|\leq
\alpha_{\rho^u,\pi_{cusp}} 
\ \& \
 1\leq i\leq n
\}.
$$
 Then holds

\medskip

(a) \qquad \qquad \qquad$\alpha_1-
\alpha_{\rho^u,\pi_{cusp}}^{(\pi)}\leq 1$ if
$\ell
\geq 1$.

\medskip

(b) \qquad \qquad \qquad
$
\alpha_i\leq \alpha_{\rho^u,\pi_{cusp}}^{(\pi)} +i;\quad
i=1,\dots,\ell.
$

\end{enumerate}

\end{proposition}

\proof
(1) Consider some $\rho_k$ such that $\rho^u\not\cong\tilde{\rho^u}$. Then (i) of Theorem 3.2 from \cite{T-CJM} and 
Theorem 7.5 of
 \cite{T86} imply that
$\pi$ is a subquotient  of a representation of the form
\begin{equation}
\label{first}
u(\d([\nu^{-(\ell_1-1)/2}\rho_k^u,\nu^{(\ell_1-1)/2}\rho_k^u]),\ell_2) \times
\rho_1'\times\dots\times\rho_i'\r\pi_{cusp},
\end{equation}
or of the form
\begin{equation}
\label{second}
\nu^\alpha 
u(\d([\nu^{-(\ell_1-1)/2}\rho_k^u,\nu^{(\ell_1-1)/2}\rho_k^u]),\ell_2)
\times
u(\d([\nu^{-(\ell_1-1)/2}\rho_k^u,\nu^{(\ell_1-1)/2}\rho_k^u]),\ell_2)
\times
\rho_i''\times\dots\times\rho_j''\r\pi_{cusp},
\end{equation}
where all the representations  $\rho_{i'}'$ and $\rho_{j'}'$ are irreducible
and cuspidal,  $0<\alpha<\frac12$ and $\rho_k$ is in the cuspidal support of the first 
factor of \eqref{first} or the first two factors of  \eqref{second}. 

Denote
$\ell_0=\ell_1+\ell_2$. 
There can be more possibilities for above representations \eqref{first} or \eqref{second}. In that case, we chose above representation with maximal possible $\ell_0$.

Suppose that $\rho_k$ is coming from \eqref{first}. 
First consider the case when $\ell_0$ is even. Then the
absolute values of the exponents  of cuspidal representations that show up in the cuspidal support of the first factor of \eqref{first} are 
containing the following sequence
$$
\textstyle
0,1,1,2,2,\dots
\frac{\ell_0-1}2-2, \frac{\ell_0-1}2-2,
\frac{\ell_0}2-1,\frac{\ell_0}2-1.
$$
(i.e. the exponents which are $\geq 1$ show up  at least two times each, and the exponent 0 shows up at least  once). 
Observe that if
we denote the above sequence by $\beta_1,\dots,\beta_l$, then 
$
\beta_i=
\frac i2
$
 for even indexes. Clearly, $|\b_1|=0\leq \frac12$. 
For odd $i>1$ we obviously have $|\b_i|=|\b_{i-1}|=\frac{i-1}2< \frac i2$.

Obviously $|e(\rho_k)|\leq \b_k$ if $k\leq \ell_0$. Therefore, then holds $|e(\rho_k)|\leq \frac k2$. If $k>\ell_0$, then the maximality of $\ell_0$ implies $|e(\rho_k)|\leq |e(\rho_{\ell_0})|\leq \frac {\ell_0} 2\leq \frac k2.$ This completes the proof of  the claim of (1) in this case.

Similarly goes the case when $\ell_0$ is odd. Then
absolute values of the exponents that show up in the case of the first factor of \eqref{first} now contain the sequence
$$
\textstyle
\frac12,\frac12,\frac32,\frac32,\dots
\frac{\ell_0}2-2, \frac{\ell_0}2-2,
\frac{\ell_0}2,\frac{\ell_0}2.
$$
Again we denote the above sequence by $\beta_1,\dots,\beta_l$, and now  holds
$
\beta_i=
\frac i2 $
if $i$ is odd. For even $i$ we have as above $|\b_i|=|\b_{i-1}|=\frac{i-1}2< \frac i2$.
Now we finish the proof of (1) in this case in the same way as in the previous case.

Suppose now that $\rho_k$ is in the cuspidal support  of the first two factors in
\eqref{second}. Consider first the case of even $\ell_0$. In this case the first two factors of \eqref{second} give exponents which contain the
following sequence
$$
\textstyle
\alpha,1-\alpha,1+\alpha,2-\alpha,2+\alpha,\dots,
\frac{\ell_0}2-1-\alpha,\frac{\ell_0}2-1+\alpha.
$$ 
Observe that each term of the above sequence is less then the corresponding term of the following sequence $\frac12,1,\frac32,2,\dots,\frac{\ell_0}2-1,\frac{\ell_0}2-\frac12.$ Now we end the proof in the same way as in previous cases.

 At the end consider the case of odd $\ell_0$. In this case the first two  factors
of \eqref{second} give the a sequence of exponents that contains the following exponents
$$
\textstyle
\frac12-\alpha,\frac12+\alpha,\frac32-\alpha,\frac32\alpha,2+\alpha,\dots,
\frac{\ell_0-1}2-\alpha,\frac{\ell_0-1}2+\alpha.
$$ 
Again we use the sequence $\frac12,1,\frac32,2,\dots,\frac{\ell_0}2-1,\frac{\ell_0}2-\frac12$ in the same way as above, and complete the proof.

\noindent 
(2) The first inequality in (2), $\a_i-\a_{i-1}\leq 1$, is proved in Proposition 3.5 of \cite{T-CJM}.

Now we shall first prove the first inequality in (ii). Suppose that for all $\rho_i$ holds $|e(\rho_i)|>
\alpha_{\rho^u,\pi_{cusp}}$ (we here assume $\alpha_{\rho^u,\pi_{cusp}} \geq\frac12$). Denote
$$
\rho_i'=
\begin{cases}
\tilde \rho_i, \quad e(\rho_i)<0
\\
\rho_i, \quad e(\rho_i)>0.
\end{cases}
$$
 Then by (ii) of Proposition \ref{CJM}, we can write $\pi\cong\tau\rtimes \pi'$, where cuspidal support
of $\tau$ is contained in $\{\rho_i', \dots, \rho_n'\}$,
and no one of the representations of $\{\rho_i', \dots, \rho_n'\}$ or their contragredients is a factor of
$\pi'$. Moreover, for any factor $\mu$ of $\pi'$ we have $\mu^u\not\cong
\rho^u$. We denote the Hermitian contragredient $\tilde{\bar \pi}$ by $\pi^+$. Since $\pi$ is unitarizable, it is a Hermitian representation, i.e. holds we $\pi\cong \pi^+$. Thus $\tau\rtimes \pi'\cong \tau^+\rtimes {\pi'}^+\cong
\bar\tau\rtimes {\pi'}^+.$

Observe that a representation in the cuspidal support of $\tau$ is of the form $\nu^\a\rho^u$, with $\a>0$. Now the complex conjugate of ${\nu^\a\rho^u}$ is isomorphic to $\nu^\a\bar{\rho^u}\cong \nu^\a\bar{\tilde{\rho^u}}\cong \nu^\a{\rho^u}$. From this follows that $\tau$ and $\bar\tau$ have the same cuspidal supports. Now the unicity claimed in (ii) of Proposition \ref{CJM} implies $\tau\cong\bar\tau$ and $\pi'\cong{\pi'}^+$.

Further,  by (i) of Proposition \ref{CJM}  $\nu^\alpha\tau\rtimes\pi'$ is irreducible for any
$\alpha\geq0$. Also $(\nu^\alpha\tau\rtimes\pi')^+\cong (\nu^\alpha\tau)^+\rtimes{\pi'}^+ \cong (\nu^\alpha\tau)^-\rtimes{\pi'} \cong \nu^\alpha\bar\tau\rtimes{\pi'}\cong \nu^\alpha\tau\rtimes{\pi'}$. Thus, $\nu^\a\tau\rtimes\pi'$, $\a\geq0$ is a family of irreducible Hermitian representations. Since it is continuous family and for $\a=0$ we have unitarizability, the whole family consists of unitarizable representations (see the construction (b)
from the third section of \cite{T-ext}).
This is impossible since this family is
not bounded (see \cite{T88}).  
Thus,
$|e(\rho_i)|\leq \alpha_{\rho^u,\pi_{cusp}}$ for for  at least one index  $i$. 

The proof of (a) goes in a similar way (we suppose that (a) does not hold,
and then we can construct complementary series which go to infinity, which is impossible). Now (b) follows from (a), and the first inequality in (2) ($\a_i-\a_{i-1}\leq 1$).

For (i), we first prove $\a_i\leq 1/2$ (for the proof, we suppose $\a_i> 1/2$, and form as above complementary series which go to infinity, which is impossible). After this, the first inequality of (2) implies the rest of (i).
\endproof

From the last proposition we see that one of the key information for understanding where unitarizability can show up in the parabolically induced representations is contained in bounds that we can get on cuspidal reducibility points $\a_{\rho,\pi_{cusp}}$. Now we shall turn our attention to that reducibility points.

One could get a bound in the following way. If the reducibility point is strictly positive, then we have a complementary series. The end of complementary series is a representation of length two, and both irreducible subquotients are unitarizable. Therefore, they have bounded matrix coefficients. Now using Casselman's asymptotics of matrix coefficients, one would get an explicit bound for the reducibility point. This bound might not be very accurate. One can get much more accurate estimate using the recent work of J. Arthur and C. M\oe glin. We shall use this approach. The references to their work, what we shall use, are now complete. A first general consequence of their work is that always
$$
\textstyle
\a_{\rho,\pi_{cusp}}\in (1/2)\Z.
$$

Now we shall recall of definition of Jordan blocks $Jord(\s)$ of an irreducible square integrable representation $\s$ of $S_q$ (slightly differently defined then originally by C. M\oe glin). In $Jord(\s)$ are irreducible selfdual square integrable representations of general linear groups. Such a representation $\tau=\delta(\rho,k)$ belongs to $Jord(\s)$ if and only if
$$
\d(\rho,k)\r\s
$$
is irreducible, and
$$
\d(\rho,l)\r\s
$$
is reducible for some $l$ of the same parity as $k$ (C. M\oe glin considers instead of a representation $\tau=\delta(\rho,k)$, pair $(\rho,k)$ which parameterize the square integrable representation).

In the rest of this section we assume additionally that
$$
\text{char}(F)=0.
$$
C. M\oe glin has proved that for an irreducible square integrable representation $\sigma$ of $S_q$ holds
\begin{equation}
\label{dimension}
\sum_{\d(\rho,k)\in Jord(\sigma)} k n_\rho=q^*,
\end{equation}
where $n_\rho$ is determined by the condition that $\rho$ is a representation of $GL(n_\rho,F)$, and further
where $q^*$ is the dimension of the vector space on which the dual group $^L(S_q)^{\hskip1pt0}$ acts (for  $Sp(2q,F)$, it is $q^*=2q+1$, and $2q$ in the case of $SO(2n+1,F)$).
For fixed $\rho$, $k$'s such that $\d(\rho,k)\in Jord(\s)$ are always of the same parity. 

Denote
$$
\text{max}_\rho(\s)=\max\{k;\delta(\rho,k)\in Jord(\sigma)\}
$$
in the case if the set on the right hand side  is non-empty. Otherwise, $\text{max}_\rho(\s)$ is not defined.

C. M\oe glin has proved that if $\s$ is cuspidal and $\d(\rho,k)\in Jord(\s)$ is such that $k\geq 3$, then $\d(\rho,k-2)\in Jord(\s)$.
Taking this into account, the equality \eqref{dimension} for cuspidal representation  $\s$ of  $S_q$ becomes
\begin{equation}
\label{main equation}
\textstyle
\sum_{\rho\,;\, \max_\rho(\sigma)\in 2\Z} \frac{\text{max}_\rho(\s)(\text{max}_\rho(\s)+2)}4 n_\rho+
\sum_{\rho\,;\, \max_\rho(\sigma)\in 1+2\Z} \frac{(\text{max}_\rho(\s)+1)^2}4 n_\rho=q^*.
\end{equation}
Now the basic assumption tells $\alpha_{\rho,\sigma}=(\text{max}_\rho(\s)+1)/2$ if $\max_\rho(\s)$ is defined (see \cite{Moe-T}). Then.
 $\text{max}_\rho(\s)=2\alpha_{\rho,\sigma}-1$, and this  implies
$$
\textstyle
\sum_{\rho\, ; \, \alpha_{\rho,\sigma}\geq 1 \, ; \, \alpha_{\rho,\sigma}\in\Z} (\alpha_{\rho,\sigma}^2-\frac14) n_\rho
 + \sum_{\rho\, ; \, \alpha_{\rho,\sigma}\geq 1 \, ; \, \a\not\in\Z}
 \alpha_{\rho,\sigma}^2 n_\rho
 =q^*.
$$
This directly implies the following

\begin{lemma}
\label{cuspidal bound} Let $\rho$ be an irreducible cuspidal self dual representation of $GL(p,F)$ and let $\sigma$ be an irreducible cuspidal representation of $S_q$ such that  $\a_{\rho,\s}\geq 1$. Then
$$
\textstyle
\alpha_{\rho,\s}^2\leq 
\begin{cases}
\frac{q^*}{p}, & \a_{\rho,\s}\in\Z;
\\
\frac{q^*}p+\frac14,  & \a_{\rho,\s}\not\in\Z.
\end{cases}
$$
\end{lemma}

\begin{definition}
Let $\pi$ be an irreducible  representation of $S_n$. Then $\pi$ is a subquotient of a representation of the form
$$
\rho_1\times\dots\times\rho_k\rtimes \sigma,
$$
where $\rho_i$ are irreducible cuspidal representations of general linear groups and $\sigma$ is an 
irreducible cuspidal representation of some $S_q$
such that $e(\rho_1)\geq e(\rho_2)\geq \dots\geq e(\rho_k)\geq 0$.
The $n$-tuple
$$
(\underbrace{e(\rho_1),\dots,e(\rho_1)}_{n_{\rho_1}-\text{times}},\underbrace{e(\rho_2), \dots,e(\rho_2)}_{n_{\rho_2}-\text{times}}, \dots, \underbrace{e(\rho_k), \dots,e(\rho_k)}_{n_{\rho_k}-\text{times}},\underbrace{0,\dots,0}_{q-\text{times}})
$$
is uniquely determined by $\pi$, and it is denoted by
$$
||\pi||.
$$
\end{definition}

The trivial (one-dimensional) representation of a group $G$ will be denoted by $\mathbf 1_G$.

Bellow we shall restrict to the groups $Sp(2n,F)$ and $SO(2n+1,F)$ (but we expect that the following statements also hold for general classical groups, what should be relatively easy too check).

\begin{lemma}
Let $\rho $  be an irreducible selfual cuspidal representation of $GL(p,F)$ and let  $\sigma$ be an irreducible cuspidal representation of $Sp(2q,F)$ or $SO(2q+1,F)$.  Consider the subgroup $X:=\{(a,1,1,\dots,1,1,a^{-1});a\in F^\t\}$ of $S_{q+1}$. Denote $e_q:=e(\d_{P_{min}}^{1/2}|X)$, where $\d_{P_{min}}$ denotes the modular character of $\d_{P_{min}}$. Then
$$
\alpha_{\rho,\s}\leq e_q.
$$
\end{lemma} 

\begin{proof} Consider first the case of symplectic groups. Then $e_q=q+1$. If $\alpha_{\rho,\s}\in \Z$, then by Lemma \ref{cuspidal bound} holds $\alpha_{\rho,\s}^2\leq \frac{2q+1}p\leq2q+1\leq(q+1)^2$, which implies the inequality in the lemma.
 If $\alpha_{\rho,\s}\not\in \Z$, then by Lemma  \ref{cuspidal bound} holds $(\alpha_{\rho,\s}^2-\frac12)p\leq2q$ since $\alpha_{\rho,\s}^2-\frac12$ is even number in that case. Thus $\alpha_{\rho,\s}^2\leq\frac{2q}p+\frac12\leq (q+1)^2$, which completes the proof in this case. 
 
In the case of odd orthogonal groups we have $e_q=q+\frac12$. Now Lemma  \ref{cuspidal bound} implies $\alpha_{\rho,\s}^2\leq \frac{2q}p+\frac14\leq2q+\frac14\leq(q+\frac12)^2$, which again implies the inequality in the lemma.
\end{proof}

\begin{theorem}
\label{bound by trivial} Suppose char$(F)=0$.
Let $\pi$ be an irreducible unitarizable representation of $G=Sp(2n,F)$ or $G=SO(2n+1,F)$. Then 
$$
||\pi||\leq_s ||\mathbf 1_G||,
$$
and the equality holds if and only if $\pi$ is a twist by a character of the trivial representation or a twist by a character of the Steinberg representation.

\end{theorem}

\begin{proof}
If $\pi$ is cuspidal, the theorem obviously holds (in particular, the theorem holds for $n=0$). It remains to consider the case of  non-cuspidal representation  $\pi$ of some $S_n$, $n\geq 1$. 

Let $\pi$ be a (non-cuspidal)  unitarizable irreducible subquotient of $\rho_1\t\dots\t\rho_k\r\s$ such that $e(\rho_1)\geq e(\rho_2)\geq \dots\geq e(\rho_k)\geq 0$. Fix some $\rho_{i_0}$, and let $\rho_1',\dots,\rho_{k'}'$ be a subsequence of all $\rho_i$ such that $\rho_i^u\cong \rho_{i_0}^u$ (we continue to  assume $e(\rho_1')\geq e(\rho_2')\geq \dots\geq e(\rho_{k'}')$).

Suppose $\rho_{i_0}\not\cong\tilde\rho_{i_0}$. Then (1) of Proposition \ref{bounds-general} obviously implies
$$
\textstyle
(\underbrace{e(\rho_1'),\dots,e(\rho_1')}_{n_{\rho_1'}-\text{times}},\underbrace{e(\rho_2'), \dots,e(\rho_2')}_{n_{\rho_2'}-\text{times}}, \dots, \underbrace{e(\rho_{k'}'), \dots,e(\rho_{k'}')}_{n_{\rho_{k'}'}-\text{times}}
) \leq_s (\frac r2,\frac{r-1}2,\dots,\frac32,1,\frac12)
$$
for appropriate $r$.

Suppose now  $\rho_{i_0}\cong\tilde\rho_{i_0}$. Since by the above lemma  $\alpha_{\rho,\s}\leq e_q$, now (2) of Proposition \ref{bounds-general} obviously implies
$$
(\underbrace{e(\rho_1'),\dots,e(\rho_1')}_{n_{\rho_1'}-\text{times}},\underbrace{e(\rho_2'), \dots,e(\rho_2')}_{n_{\rho_2'}-\text{times}}, \dots, \underbrace{e(\rho_{k'}'), \dots,e(\rho_{k'}')}_{n_{\rho_{k'}'}-\text{times}}
,\underbrace{0,\dots,0}_{q-\text{times}}
) 
\hskip30mm
$$
$$
\hskip90mm
\leq_s (r,r-1,\dots, \e+2,\e+1,\e)
$$
for appropriate $r$ and $\e=1$ (resp. $\e=\frac12$) if $S_n=Sp(2n,F)$ (resp. $S_n=SO(2n+1,F)$).

The two above relations directly imply the inequality of the theorem.

Regarding equality, let us first consider the symplectic case.
To get the equality, in all the above sequence of inequalities we must have always equalities. This implies  that we must have  $q=0$,  $n_{\rho_i}=1$ for all $i$, $k=k'$ and $\rho_k=\nu\mathbf1_{F^\times}$. This further implies that we must have $\rho_i=\nu^{k+1-i}1_{F^\times}$ for all the other indexes. At the corresponding induced representation we have precisely two unitarizable subquotients by  \cite{Ca81}, the trivial and the Steinberg representation.

Similarly in the odd-orthogonal case, to get equalities at all the steps, we must have $q=0$,  $n_{\rho_i}=1$ for all $i$, $k=k'$, $\rho_k=\e\nu\mathbf1_{F^\times}$, with $\e^2\equiv 1$. This  further implies $\rho_i=\nu^{k+1-i}\e \mathbf 1_{F^\times}$ for all the other indexes. At the corresponding induced representation we have precisely two unitarizable subquotients, the irreducible quotient and the irreducible subrepresentation (besides \cite{Ca81}, see also \cite{HT} and \cite{HJ}).
\end{proof}

It is evident from the  proof  of the above theorem that we can give much more accurate upper bound if we know by which parabolic subgroup is supported irreducible unitary representation.
The following theorem is a result in that direction, which has the same proof as the previous theorem:

\begin{theorem} Assume char$(F)=0$.
Let $\pi$ be an irreducible unitarizable representation of $S_n$ supported by a parabolic subgroup whose Levi factor is isomorphic to
$$
GL(p_1,F)^{n_1}\t\dots GL(p_k,F)^{n_k}\t S_q,
$$ 
where $n_i\ne n_j$ for $i\ne j$.
Let $\pi$ be an irreducible unitarizable subquotient of $\rho_1\t\dots\t\rho_k\r\s$. We chose $\rho_i$ such that all $e(\rho_i)\geq0$. Fix some index ${i_0}$, and let $\rho_1',\dots,\rho_{n_{i_0}}'$ be a subsequence of all $\rho_i$  which are representations of $GL(n_{i_0},F)$. After a renumeration, we can assume $e(\rho_1')\geq e(\rho_2')\geq \dots\geq e(\rho_{n_{i_0}}')(\geq 0)$.
Then  
$$
\textstyle
(e(\rho_1'),e(\rho_2'), \dots,e(\rho_{i_0}'))
 \leq_s (r,r-1,\dots,c+1,c)
$$
for appropriate $r$, where
$$
\textstyle
c=\max\{t\in (1/2)\Z; t\leq \sqrt{\frac{q^*}{n_{i_0}}+\frac14}\}.
$$

\end{theorem}

\begin{example}
 Consider an example of the symplectic group where $k=1$, $p_1=2$, $n_1=5$ and  $q=6$. The above theorem gives the bound $(\frac{13}2,\frac{11}2,\frac92,\frac72,\frac52)$. In other words, 
$$
\textstyle
||\pi||\leq_s (\frac{13}2,\frac{13}2,\frac{11}2,\frac{11}2,\frac92,\frac{9}2,\frac72,\frac{7}2,\frac52,\frac{5}2,0,0,0,0,0,0).
$$
This is much sharper estimate then given by Theorem \ref{bound by trivial}, 
 which gives the bound 
$$
||\pi||\leq_s (16,15,\dots,2,1).
$$

\end{example}

At the end, the following theorem gives upper bounds for individual Bernstein components.

\begin{theorem} 
\label{thm-fixed-component} Let char$(F)=0$.
Fix an  irreducible cuspidal representation $\sigma$ of $S_q$. Let $\rho_i$ be  irreducible unitarizable cuspidal representations of $GL(p_i,F)$, $i=1,\dots,k$ such that $\rho_i\not\cong\nu^\b\rho_j$ for any $i\ne j$ and any $\b\in \mathbb C$, and let $n_1,\dots,n_k$ be positive integers.

Let $\b_{i,j}, 1\leq i\leq k, 1\leq j\leq n_i$ be a set of complex numbers such that
the representation
 $$
\nu^{\b_{1,1}}\rho_1\t\dots \t\nu^{\b_{1,n_1}}\rho_1\t\dots
$$
 $$
\dots\t \nu^{\b_{k,1}}\rho_1\t\dots \t\nu^{\b_{k,n_k}}\rho_1\r\s
$$
contains an irreducible unitarizable subquotient.
  Fix some index ${i}$. After a renumeration we can assume that for real parts of complex exponents hold $|\Re(\b_{i,1})| \geq  \dots\geq |\Re(\b_{i,n_i})|$. Denote by $X_i$ the set of all unramified  characters $\chi$ of $GL(p_i,F)$ such that $\chi\rho_i$ is a self dual representation. Then $X_i$ is a finite set. 
   If $X_i\ne\emptyset$, set
$$
c_i=\frac{1+\max\{\text{card}(Jord_{\chi\rho_i}(\s));\chi\in X_i\}}2.
$$
Then
$$
\textstyle
(|\Re(\b_{i,1})|,  \dots, |\Re(\b_{i,n_i})|)
 \leq_s (c_i+n_i-1,c_i+n_i-2,\dots,c_i+1,c_i)
$$
 if $X_i\ne\emptyset$, and 
$$
\textstyle
(|\Re(\b_{i,1})|,  \dots, |\Re(\b_{i,n_i})|)
 \leq_s (\frac {n_i}2,\frac{n_i-1}2,\dots,1,\frac12)
$$
if $X_i=\emptyset$.

\end{theorem}

\section{Unramified unitary dual of $Sp(2n,F)$}

In this and the following section, $F$  denotes  a local non-archimedean field satisfying
$$
\text{char}(F)\ne2.
$$
 The ring of integers in $F$ is 
denoted by $\mathcal O_F$. We fix an uniformizing element of $\mathcal O_F$ and denote it
 by
$\varpi_F$. The normalized absolute value on $F$ is denoted by $|\
\ |_F$. Then  $|\varpi_F |_F=\text{card}(\mathcal O_F/\varpi_F\mathcal
O_F)^{-1}$.

Using the determinant
homomorphism, we identify characters of $F^\times=GL(1,F)$ with
characters of $GL(n,F)$. If $\varphi$ is a character of $GL(n,F)$, then
there exist a unique unitary character $\varphi^u$ of $GL(n,F)$ and
$e(\varphi)\in \mathbb R$ such that
$$
\varphi=\nu^{e(\varphi)}\varphi^u.
$$

The subgroup of all diagonal matrices in $Sp(2n,F)$ will be denoted by $A$. The mapping diag$(a_1,\dots,a_n,a_n^{-1},.\dots,a_1^{-1}) \mapsto (a_1,\dots,a_n)$ is an isomorphism of $A$ and $(F^\t)^n$, and using this isomorphism we identify these two groups.
The subgroup of all upper triangular unipotent
matrices in $Sp(2n,F)$ will be denoted by $N$. 

In $GL(n,F)$  we fix the maximal compact subgroup
$GL(n,\mathcal O_F)$ and in $Sp(2n,F)$
 the maximal compact subgroup $K_{\text{max}}\!=\!Sp(2n,F)\,\cap\,
GL(2n,\mathcal O_F)$is fixed.   An irreducible representation
$(\pi,V)$ of
$GL(n,F)$ or $Sp(2n,F)$ is called unramified if $V$ contains 
a non-trivial vector invariant for the action of the fixed maximal compact
subgroup. Then the space of
invariant vectors for the maximal compact subgroup is one dimensional.

For the group $G(F)$ of $F$-rational points of a reductive group $G$ defined
over
$F$, we denote the set of equivalence classes of irreducible smooth
representations by $\widetilde {G(F)}$. The subset of unitarizable classes
in
$\widetilde {G(F)}$ is denoted by $\widehat {G(F)}$. If a maximal compact
subgroup in
$G$ is fixed, then we denote by $\widetilde {G(F)}^{\mathbf 1}$ the set of all
unramified classes in $\widetilde {G(F)}$. We denote by $\widehat {G(F)}^{\mathbf 1}$
the unramified classes in $\widehat {G(F)}$, and call it unramified unitary
dual.

If we have  a smooth representation $\pi$ of $Sp(2n,F)$,
we denote by 
$$
s_{(1,\dots,1,0)}(\pi)
$$
the (normalized) Jacquet module of $\pi$ with respect  to
$P_{min}=AN$. It is a representation of
$A$, which we have identified  with $(F^\times)^n$. If $\tau$ is an irreducible subquotient of
$s_{(1,\dots,1,0)}(\pi)$, using the above  identification,  we can
write
$\tau$ as
$\tau_1\otimes\dots\otimes\tau_n$, where $\tau_i$ are
characters of $F^\t$. 
Now we shall recall of some  definitions from
\cite{Mu-no-un} in the case of $Sp(2n,F)$.

\begin{definition} Let $\pi$ be an irreducible unramified
representation of $Sp(2n,F)$. Then $\pi$ is called 
negative if for any irreducible subquotient
$\varphi=\varphi_1\otimes\dots\varphi_n$ of the
Jacquet module $s_{(1,\dots,1,0)}(\pi)$ 
we have
$$
\aligned
&e(\varphi_1)\leq 0,
\\
&e(\varphi_1) + e(\varphi_2)\leq 0,
\\
&\hskip20mm\vdots
\\
&e(\varphi_1) + e(\varphi_2) + \ \dots\ + e(\varphi_n)\leq 0.
\endaligned
$$
Further, $\pi$ will be called strongly negative if above all the  inequalities are
 strict.
\end{definition}

By 
$$
\text{Jord$_{\text{sn}}'(n)$}
$$
 will be denoted the collection  of all possible finite
sets
$J:=\{(\chi_1,m_1),\dots,(\chi_k,m_k)\}$\footnote{One  possibility would be to write instead of pairs $(\chi_i,m_i)$ unramified self dual characters $\chi_i\circ\det_{m_i}$ of $GL(m_i,F)$.} such that $\chi_i$ are self dual unramified  characters of $F^\t$ and $m_i$ are odd positive integers which satisfy 
$$
\textstyle
\text{$\sum_{i=1}^k m_i=2n+1$}.
$$

There are precisely two self dual unramified  characters of $F^\t$, the trivial character $\mathbf 1_{F^\t}$ and the non-trivial self dual unramified  character, which we denote by $\mathbf{sgn}_{F^\times}$.
For a self dual unramified  character $\chi$ of $F^\t$ and $J
=\{(\chi_1,m_1),\dots,(\chi_k,m_k)\}
\in \text{Jord}_{\text{sn}}'(n)$ we denote 
$$
J(\chi)=\{m_i;\chi_i=\chi\}.
$$
If we write $J(\chi)$ for $J\in $ Jord$_{\text{sn}}'(n)$, then $\chi$ will
be always assumed to be unramified selfdual character of $F^\times$.

We denote by
$$
\text{Jord$_{\text{sn}}(n)$}
$$
the set of all $J\in $ Jord$_{\text{sn}}'(n)$ such that $J(\mathbf{sgn}_{F^\times})$ has even cardinality.

Denote
$$
 {J}(\chi)'=
\begin{cases} 
\hskip9mm{J}(\chi), \hskip35mm \text{if $\chi=\mathbf{sgn}_{F^\times}$;}
\\
{J}(\chi)\cup\{-1\}, \hskip33mm\text{if $\chi=\mathbf{1}_{F^\times}$.}
\end{cases}
$$

To a character
$\chi$ of $F^\times$ and $r_1,r_2\in\mathbb R$ such that
$r_2-r_1\in \mathbb Z$, we attach representation
$$
\langle[\nu^{r_1}\chi,\nu^{r_2}\chi]\rangle:=
\nu^{(r_2+r_1)/2}\chi
\ \mathbf1_{GL(r_2-r_1+1,F)}
$$
if $r_2\geq r_1$ (we  use here  Zelevinsky notation:
$\langle[\nu^{r_1}\chi,\nu^{r_2}\chi]\rangle$ is characterized
as a unique irreducible subrepresentation of
$\nu^{r_1}\chi\times\nu^{r_1+1}\chi
\times\dots\times\nu^{r_2}\chi$). Otherwise, we take
$\langle[\nu^{r_1}\chi,\nu^{r_2}\chi]\rangle$ to be the trivial representation of the trivial group
$GL(0,F)$ (we consider formally this group as
$0\times0$~-~matrices).

\bigskip

For
$J\in $ Jord$_{\text{sn}}(n)$ write
$J(\chi)'=\{a_{2l_\chi}^{(\chi)},a_{2l_\chi-1}^{(\chi)},\dots,a_1^{(\chi)}\}$,
where
$$
a_{2l_\chi}^{(\chi)}>a_{2l_\chi-1}^{(\chi)}>\dots>a_1^{(\chi)}
$$
(if $J(\chi)=\emptyset$ we take $l_\chi=0$).
We define $\sigma({J})$ to be  the unique
irreducible unramified subquotient of
$$
\left(\underset{\chi}
\times\left(
\overset {l_\chi} 
{\underset{i=1}
\times} \
\langle[\nu^{-(a_{2i}^{(\chi)}-1)/2}\chi,
\nu^{(a_{2i-1}^{(\chi)}-1)/2}\chi]\rangle
\right)\right) \rtimes\mathbf1_{Sp(0,F)},
$$
where the first product runs over (two) unramified selfdual characters of
$F^\times.$ 

G. Mui\'c in \cite{Mu-no-un} has proved  the following explicit
classifications of strongly negative and negative irreducible unramified
representations:

\begin{theorem} (Mu2]) (i)
The mapping $J\mapsto \sigma(J)$ is a bijection
from Jord$_{\text{sn}}(n)$ on the set of all strongly negative
irreducible unramified representations of $Sp(2n,F)$.

\medskip
\noindent (ii)  
Suppose  $J \in
\text{Jord}_{\text{sn}}(m)$ and suppose that $\psi_1,\dots,\psi_l$ are 
unramified unitary characters
 of $GL(n_1,F),\dots,GL(n_l,F)$
respectively, such that $n_1+\dots+n_l+m=n$. Let
$\pi$ be the unique unramified irreducible subquotient (actually
subrepresentation) of
$$
\psi_1\times\dots\times\psi_l\rtimes\sigma({J}).
$$
Then $\pi$ is an
irreducible negative unramified representation
 of 
$G_n(F)$. Moreover, $\pi$ determines $J$
uniquely, and it determines characters $\psi_1,\dots,\psi_l$ up to a permutation
and changes
$\psi_i\leftrightarrow \psi_i^{-1}$. Further, each irreducible negative
unramified representation of $G_n(F)$ is equivalent to some representation
$\pi$ as above.
\end{theorem}

\begin{remark} Sometimes is more
convenient the following description of $\text{Jord}_{\text{sn}}(n)$.
Since there are exactly two selfdual unramified characters of
$F^\times$, $\mathbf 1_{F^\times}$ and $\mathbf{sgn}_{F^\times}$  (the
non-trivial unramified character of order two), to $J \in
\text{Jord}_{\text{sn}}(n)$ we  attach the ordered pair
$$
(J(\mathbf 1_{F^\times}),J(\mathbf{sgn}_{F^\times})),
$$
where we consider $J(\mathbf 1_{F^\times})$ and $J(\mathbf{sgn}_{F^\times})$
as partitions. This pair determines $J$, and the partitions satisfy the following
properties.

For a partition $p$ of $n$ into sum of $k$ positive integers we shall
write $\vdash$$(p)=n$ and card$(p)=k$. We shall write always members of
partitions in descending order.

In this way, $\text{Jord}_{\text{sn}}(n)$ (and irreducible unramified
strongly negative representations of $Sp(2n,F)$) are parameterized by pairs
$$
(t,s),
$$
where both $t$ and $s$ are partitions into different odd numbers, which satisfy $\vdash$$(t)$ $+$ $\vdash$$(s)=2n+1$ and 
card$(s)\in 2\mathbb Z$.
The corresponding strongly negative representation will be denoted by
$\sigma(t,s)$.

\end{remark}

\medskip

From \cite{MuT} we get the following description of the
unramified unitary dual:

\begin{theorem}
\label{theorem-class}
(i) Let  $\varphi_i$ be  unramified
characters of 
$GL(n_i,F)$ such that $e(\varphi_i)>0$ for $i=1,\dots,m$, and let 
$\sigma_{neg}$ be an irreducible negative unramified representation of 
$Sp{(2(n-n_1-\dots-n_m)},F)$ (we assume $n_1+\dots+n_m\leq n$).
Denote
$$
\pi=
\varphi_1\times\dots\times\varphi_m\rtimes\sigma_{neg}.
$$
For any 
$\varphi$ 
showing up among $\varphi_1^u,\dots,\varphi_m^u$, denote by
$\mathbf e_\pi(\varphi)$ the multiset of exponents
$e(\varphi_i)$ for those $i$ such that $\varphi_i^u\cong\varphi$,
and suppose that 
 the following conditions hold:

\begin{enumerate}

\item
$\mathbf e_\pi(\tilde{\varphi})=\mathbf e_\pi(\varphi)$.

\item 
If either $\varphi\ne\tilde{\varphi}$, or
$\varphi=\tilde\varphi$ and $\nu^\frac12\varphi\rtimes\mathbf1_{Sp(0,F)}$ 
$\underline{reduces}$, then
$\alpha<\frac12$ for all $\alpha\in \mathbf e_\pi(\varphi)$.

\item  
If $\tilde{\varphi}\cong\varphi$ and
$\nu^\frac12\varphi\rtimes\mathbf1_{Sp(0,F)}$ is $\underline{irreducible}$,
then all exponents in $\mathbf e_\pi(\varphi)$ are $<1$. If we write
$\mathbf e_\pi(\varphi)=\{\alpha_1,\dots,\alpha_k,\beta_1,\dots,\beta_l\}$
in a way  that
$$
0<\alpha_1\le\dots\le\alpha_k\le\frac12<\beta_1\leq\dots\leq\beta_l<1,
$$
then first $\beta_1<\dots<\beta_l$ (we can have $k=0$ or $l=0$). Further
\smallskip
\begin{enumerate}
\item[(a)] $\alpha_i+\beta_j\ne1$ for all $i=1,\dots,k$,
$j=1,\dots,l$ and $\alpha_{k-1}\ne\frac12$ if $k>1$.
\item[(b)]  ${\text{card}}\big(\{1\le i\le k:\alpha_i>1-\beta_1\}\big)$ is
even if $l>0$.
\item[(c)]  ${\text{card}}\big(\{1\le i\le
k:1-\beta_j>\alpha_i>1-\beta_{j+1}\}\big)$ is odd for $j=1,\dots,l-1$.
\item[(d)]  $k+l$ is even if $\varphi\rtimes\sigma_{neg}$ 
$\underline{reduces}$. 
\end{enumerate}
\end{enumerate}
Then $\pi$
is an irreducible unitarizable unramified representations of $Sp(2n,F)$.

\smallskip
\noindent
(ii) If we have an irreducible unitarizable unramified
representation $\pi$ of $Sp(2n,F)$, then  there exist 
$\varphi_1,\varphi_2,\dots,\varphi_m,\sigma_{neg}$ as in (i), which satisfy
all the conditions in (i), such that
$$
\pi\cong\varphi_1\times\dots\times\varphi_m\rtimes\sigma_{neg}.
$$
 Further,
$\sigma_{neg}$ and the multiset
$(\varphi_1,\dots,\varphi_k)$ are uniquely determined by
$\pi$ up to equivalence.

\end{theorem}

\bigskip

To have an explicit classification, one needs to understand when 
$\nu^\frac12\varphi\rtimes \mathbf 1_{Sp(0,F)}$ and
$\varphi\rtimes\sigma_{neg}$ from  above theorem reduce. Since in the above
theorem
$\varphi$ is selfdual, we can write
$\varphi=\langle[\nu^{-(p-1)/2}\chi,\nu^{(p-1)/2}\chi]\rangle$ where $p\in
\mathbb Z_{>0}$ and
$\chi$ is a selfdual unramified character of $F^\times$.
 Now the reducibility is described by the following results of G. Mui\'c in
\cite{Mu-no-un}:
\bigskip

\begin{proposition}
\label{prop-red}
 Let  
$$
\varphi=\langle[\nu^{-(p-1)/2}\chi,\nu^{(p-1)/2}\chi]\rangle,
$$
 where $p\in
\mathbb Z_{>0}
$
 and
$\chi$ is a selfdual unramified character of $F^\times$.
Suppose that $\sigma_{neg}$ is an (unramified) irreducible
subrepresentation of some
$$
\psi_1\times\dots\times\psi_s\rtimes\sigma({J}),
$$
where $\psi_i$ are unitary unramified  characters of general linear
groups and $J\in\text{Jord}_{\text{sn}}(q)$, $q\geq0$.
Then

\medskip
\noindent
 (1) 
$\nu^\frac12\varphi\rtimes\mathbf1_{Sp(0,F)}$ reduces if and only if 
$p$ is even;

\smallskip
\noindent
(2)  $\varphi\rtimes\sigma_{neg}$
reduces if and only if $p$ is odd, $(\chi,p)\notin J$
and $\varphi\notin\{\psi_1,\dots,\psi_s\}$.
\end{proposition}

\section{Bounding the trivial representation from the rest of the unitary dual of $Sp(2n,F)$}

We continue to assume char$(F)\ne2$, as we did in the previous section. 
Denote
$$
 \mathbb R^q_\downarrow=\{
x=(x_1,x_2,\dots,x_{q-1},x_q) \in \mathbb R^q; x_1\geq x_2\geq \dots,x_{q-1}\geq x_q\}.
$$
For $x\in  \mathbb R^q$ we denote by 
$$
x_\downarrow
$$
 a unique $y \in \mathbb R^q_\downarrow$ such that the sequences $x_1,x_2,\dots,x_{q-1},x_q$ and $y_1,y_2,\dots,y_{q-1},y_q$ coincide up to a permutation. For $x\in \R^q$ we denote
$$
|x|=(|x_1|,|x_2|,\dots,|x_{q-1}|,|x_q|).
$$

We have defined two orderings on $ \mathbb R^q$: 
$
x\leq_wy $ if $ \sum_{i=1}^j x_i\leq \sum_{i=1}^j y_i$ for all
$j\in\{1,\dots,q\}
$,
and
$
x\leq_s y $ if $ x_j\leq  y_j$ for all
$j\in\{1,\dots,q\}.
$
Then obviously hold the following simple properties
$$
x\leq_w y
\qquad \& \qquad
x'\leq_w y' \qquad \implies \qquad x+x'\leq_w y+y',
$$
$$
x\leq_s y
\qquad \& \qquad
x'\leq_s y' \qquad \implies \qquad x+x'\leq_s y+y',
$$
$$
x\leq_s y\implies x\leq_w y,
$$
$$
x\leq_w |x|, \qquad
x\leq_s |x|,
$$
$$
x\leq_w x_\downarrow.
$$
The last inequality holds since the sum of the first $j$ coordinates of $x_\downarrow$ is greater then  or equal to the sum of any $j$ coordinates of $x_\downarrow$ (or $x$). Note that $
x\leq_s x_\downarrow.
$ does not hold in general.

For $x\in \mathbb R^q$ and $y\in \mathbb R^p$ we denote 
$$
x_-^-y=(x_1,x_2,\dots,x_{q-1},x_q,y_1,y_2,\dots,y_{p-1},y_p)\in \mathbb R^{q+p}.
$$

\begin{lemma}
\label{lemma-ineq}
  (i) Let $x\in \mathbb R^q_\downarrow$ and $y\in \mathbb R^q$. Then $x\geq_s y$ implies $x\geq_s y_\downarrow$.

\noindent(ii) For $x,x'\in \mathbb R^q_\downarrow$ and $y,y'\in \mathbb R^p_\downarrow$ holds
$$
x \ \geq_s \ x',\ y\ \geq_s \ y'\implies (x_-^-y)_\downarrow\  \geq_s \ ({x'}_-^-y')_\downarrow.
$$
\end{lemma}

\begin{proof} (i) The assumption is that $x_1\geq x_2\geq \dots,x_{q-1}\geq x_q$ and  $x_i\geq y_i$, $i=1,\dots,q$. 
Suppose that $y_i<y_j$ for some $i<j$. Denote by $y'\in \mathbb R^q$ the element which one gets from $y$ switching positions of $y_i$ and $y_j$. Now
$
x_i\geq x_j\geq y_j,
$
and 
$
x_j\geq  y_j>y_i.
$
This implies $x\geq_s y'$. 
Further, if there are some $i<j$ such that $y_i'<y_j'$, one defines $y''$ in analogous way as was defined $y'$ from $y$, and gets in the same way that $x\geq_s y''$. Repeating this procedure as long as it is possible (one can do it at most finitely many times),  one will get $x\leq_s y^{(n)}.$  Since $y_\downarrow=y^{(n)}$, the proof of is  complete.

Denote $z=x_-^-y, z'={x'}_-^-y'$. Obviously, $z\geq_s z'$. Let $\sigma$ be a permutation of $\{1,\dots,q+p\}$ such that $z_\downarrow=(z_{\sigma(1)},\dots,z_{\sigma(q+p)})$. Denote $z''=(z'_{\sigma(1)},\dots,z'_{\sigma(q+p)})$.
Observe that $z\geq_s z'$ implies $z_\downarrow\geq_s z''$. Now  (i) obviously implies $z_\downarrow\geq_s z''_\downarrow$.
Since $z'_\downarrow= z''_\downarrow$, we get $z_\downarrow\geq_s z'_\downarrow$. This completes the proof of (ii).
\end{proof}

Let  $\pi$ be an unramified irreducible representation of $GL(q,F)$. Then
$\pi$ is a subquotient of some 
$$
\chi_1\times\dots\times\chi_q,
$$
where $\chi_i$ are unramified characters of $F$.
The sequence
$
e(\chi_1),\dots,e(\chi_n)
$
is determined 
  by $\pi$ up to  a permutation.
We denote by
$$
\mathbf e(\pi)=(e(\chi_1),\dots,e(\chi_q))_\downarrow
$$

Let $\tau_1,\dots,\tau_l$ be  irreducible unramified representations of  general
linear groups, and let $\pi$ be such a representation of a classical group. Then
$\tau_1\times\dots\times\tau_l\rtimes \sigma$ contains a unique unramified irreducible subquotient.
Denote it by $\pi'$. Then we define
$||\tau_1\times\dots\times\tau_l\rtimes \sigma||$ to be
$||\pi'||$.
Observe that
$$
||\tau_1\times\dots\times\tau_l\rtimes \sigma||=
(|\mathbf e(\tau_1)|^-_-\dots^-_-|\mathbf e(\tau_l)|^-_-|| \sigma||)_\downarrow.
$$

For $u,v\in \mathbb R$ such that $v-u$ is a non-negative integer, we denote
$$
[u,v]_\downarrow=(v,v-1,,\dots,u+1,u)\in \mathbb R^{v-u+1}.
$$

\begin{proposition}  Let $\pi$ be a strongly negative unramified  representation of $Sp(2q,F)$. Then
$$
||\pi||\leq_s [1,q]_\downarrow,
$$
where the equality holds if and only if $\pi$ is the trivial representation of $Sp(2q,F)$. If $\pi$ is a non-trivial strongly negative unramified  representation, then 
$$
||\pi||\leq_s  [0,q-1]_\downarrow.
$$
\end{proposition}

\begin{proof} By \cite{Mu-no-un}, $\pi=Jord(a,b)$, for some partitions $a$ and $b$  into different odd positive integers such that $\vdash$$(a)$ + $\vdash$$(b)=2q+1$ and that the number of integers in $b$ is even. We shall prove the proposition by the induction with respect to the sum of number of integers entering partitions $a$ and $b$ (obviously, this number is always odd). We shall denote this number by $m$.

Consider first the case when that number is one. Then $(a,b)=((2q+1),\emptyset)$ and $\pi$ is the trivial representation $\mathbf1_{Sp(2q,F)}$. Obviously, $||\mathbf1_{Sp(2q,F)}||= [1,q]_\downarrow$.
 
 We go now to the inductive step. Suppose now  $m\geq 3$. Then at least one of partitions $a$ or $b$ has at least 2 integers. We shall consider the case that $a$ has at lest two integers (the other case goes completely analogously). 
 
 Write $a=(a_1,a_2)^-_-a'$ in a way that $a_1>a_2$. Then $a_1+a_2+\vdash$$(a')$ + $\vdash$$(b)=2q+1$. Define an integer $q'$ by the requirement $\vdash$$(a')$ + $\vdash$$(b)=2q'+1$.
 Now $\pi$ is a non-trivial representation, and we have
 $$
 \textstyle
 ||\pi||=||Jord(a,b)||= ||\langle[\nu^{-\frac{(a_2-1)}2}\mathbf1_{F^\times},\nu{\frac{(a_3-1)}2}\mathbf1_{F^\times}]\rangle)\rtimes  Jord(a',b))||
 $$
$$
\textstyle
=({[1,\frac{(a_2-1)}2]_\downarrow}^-_-
{[0,\frac{(a_1-1)}2]_\downarrow }^-_-
 ||Jord(a',b)|| )_\downarrow.
$$
Now using the above lemma and the inductive assumption,  we get
$$
\textstyle
||\pi|| 
\leq
({[1,\frac{(a_2-1)}2]_\downarrow}^-_-
{[0,\frac{(a_1-1)}2]_\downarrow }^-_-
 [1,q']_\downarrow )_\downarrow.
$$
$$
\textstyle
=({[1,\frac{(a_2-1)}2]_\downarrow}^-_-
 {[1,q']_\downarrow}
^-_-
{[0,\frac{(a_1-1)}2]_\downarrow }
 )_\downarrow.
$$
Again using the above lemma we get
$$
||\pi|| \leq_s [0,q-1]_\downarrow.
$$
This completes the proof of the proposition.
\end{proof}

\begin{proposition} Let $\pi$ be a  negative unramified  representation of $Sp(2q,F)$ which is not strongly negative. Then at least one of the following two inequalities hold:
$$
||\pi||\leq_s [0,q-1]_\downarrow
$$
or
$$
\textstyle
||\pi||\leq_s  {[1,q-2]_\downarrow}^-_-(\frac12,\frac12).
$$
\end{proposition}

\begin{proof} Observe that both right bounds are $\leq_s [1,q]_\downarrow$ (we shall use this evident fact below). 

We shall prove the proposition by induction with respect to the rank $q$. By \cite{Mu-no-un}, we can write $\pi=\mathfrak z([-\frac{(c-1)}2,\frac{(c-1)}2]^{(\chi)})\times\pi'$, where $\pi'$ is a negative representation of $Sp(2q',F)$ (then $2c+q'=q$), $c$ is a positive integer and $\chi$ is a unitary unramified character. 

Consider first the case of odd $c$. Then
 $$
 ||\pi||=||\textstyle\langle[\nu^{-\frac{(c-1)}2}\chi,\nu^{\frac{(c-1)}2}\chi]\rangle\rtimes \pi'||
 $$
 $$
 \textstyle
= ({[1,\frac{(c-1)}2]_\downarrow}^-_- {[0,\frac{(c-1)}2]_\downarrow}^-_-|| \pi'||)_\downarrow
 $$
Now using the above lemma and the inductive assumption, we get
$$
||\pi|| \leq_s
\textstyle
 ({[1,\frac{(c-1)}2]_\downarrow}^-_- {[0,\frac{(c-1)}2]_\downarrow}^-_-[1,q']_\downarrow)_\downarrow
$$
$$
\textstyle
 ({[1,\frac{(c-1)}2]_\downarrow}^-_- {[1,q']_\downarrow}^-_-      [0,\frac{(c-1)}2]_\downarrow)_\downarrow.
$$
Again using the above lemma  we get
$$
||\pi|| \leq_s [0,q-1]_\downarrow.
$$

Now consider the case of even $c$. Then
 $$
 \textstyle
 ||\pi||=||\langle[\nu^{-\frac{(c-1)}2}\chi,\nu^{\frac{(c-1)}2}\chi]\rangle \rtimes \pi')||
 $$
 $$
 \textstyle
= ({[\frac12,\frac{(c-1)}2]_\downarrow}^-_-{[\frac12,\frac{(c-1)}2]_\downarrow}_-^-||\pi'||)_\downarrow.
 $$
Now using the above lemma we get
$$
||\pi|| \leq_s
\textstyle
 ({[\frac12,\frac{(c-1)}2]_\downarrow}^-_-{[\frac12,\frac{(c-1)}2]_\downarrow}_-^-[1,q']_\downarrow)_\downarrow
 $$
 $$
 \textstyle
 ({[\frac32,\frac{(c-1)}2]_\downarrow}^-_-{[\frac32,\frac{(c-1)}2]_\downarrow}_-^-{[1,q']_\downarrow}^-_-(\frac12,\frac12))_\downarrow.
 $$
Again using the above lemma we get
$$
\textstyle
||\pi|| \leq_s {[1,q-2]_\downarrow}^-_-(\frac12,\frac12).
$$
This completes the proof of the proposition.
\end{proof}

\begin{proposition} Let $q\geq 2$ and $\pi$ be an irreducible  unitarizable unramified  representation of $Sp(2q,F)$ which is not  negative. Then the following  inequality holds
$$
\textstyle
||\pi||\leq_s {[1,q-1]_\downarrow}^-_-(\frac12).
$$
Further, if we write $||\pi||=(x_1,  \dots,x_q)$, then
$$
x_1+\dots+x_n\leq {q(q-1)}/2.
$$
\end{proposition}

\begin{proof} 
 Let $c$ be a positive integer and $\chi$ a unitary character. Below we shall use  the following inequalities which follow from Lemma \ref{lemma-ineq}:
$$
\textstyle
0<\alpha\leq \frac12
\implies
|\mathbf e( \nu^\alpha \langle[\nu^{-\frac{(c-1)}2}\chi,\nu^{\frac{(c-1)}2}\chi]\rangle)|_\downarrow\leq_s
\begin{cases}
\big({[\frac12,\frac c2]_\downarrow}^-_-{[1,\frac{c-1}2]}_\downarrow\big)_\downarrow, & c \text{ is odd};
\\
\big({[1,\frac c2]_\downarrow}^-_-{[\frac12,\frac{c-1}2]}_\downarrow\big)_\downarrow, & c \text{ is even};
\end{cases}
$$
$$
\textstyle
\frac12<\alpha<1 
\implies
|\mathbf e( \nu^\alpha \langle[\nu^{-\frac{(c-1)}2}\chi,\nu^{\frac{(c-1)}2}\chi]\rangle)|_\downarrow
\leq_s
\begin{cases}
\big({[1,\frac {c+1}2]_\downarrow}^-_-{[\frac12,\frac{c}2-1]}_\downarrow\big)_\downarrow, & c \text{ is odd};
\\
\big({[\frac12,\frac {c+1}2]_\downarrow}^-_-{[1,\frac{c}2-1]}_\downarrow\big)_\downarrow, & c \text{ is even}.
\end{cases}
$$
In the case $c=1$ we take $[1,0]_\downarrow=[\frac12,-\frac{1}2]_\downarrow$ to be the empty set.

Observe that the above estimates imply that for odd $c\geq 3$ and $0<\alpha<1$ holds
\begin{equation}
\label{odd}
\textstyle
|\mathbf e( \nu^\alpha \langle[\nu^{-\frac{(c-1)}2}\chi,\nu^{\frac{(c-1)}2}\chi]\rangle)|_\downarrow\leq_s
[\frac12,\frac{2c-1}2]_\downarrow.
\end{equation}

We shall prove the proposition by induction with respect to the rank $q$. 
Since $\pi$ is not negative,
 Theorem \ref{theorem-class} implies that $\pi$ is a member of a complementary series. We start with the complementary series of the form
 $$
 \textstyle
 \pi= 
 \nu^\alpha \langle[\nu^{-\frac{(c-1)}2}\chi,\nu^{\frac{(c-1)}2}\chi]\rangle
 \times 
 \nu^\alpha \langle[\nu^{-\frac{(c-1)}2}\bar\chi,\nu^{\frac{(c-1)}2}\bar\chi]\rangle
  \rtimes\pi',
 $$ 
 where $\pi'$ is a unitarizable representation of $Sp(2q',F)$ (then $2c+q'=q$), $c$ is a positive integer, $\chi$ is a unitary unramified character and $0<\a<1/2$. 
 We shall break analysis of this complementary series  into two possibilities. The first case is when $c$ is odd. Then
 $$
 \textstyle
 ||\pi||
  \textstyle
 \leq_s ({[\frac12,\frac{c}2]_\downarrow}^-_-{[\frac12,\frac{c}2]_\downarrow}^-_-
 {[1,\frac{(c-1)}2]_\downarrow}^-_-{[1,\frac{(c-1)}2]_\downarrow}^-_-[1,q']_\downarrow)_\downarrow
 $$
 The above inequality implies 
 $$
  \textstyle
 ||\pi ||\leq_s  {[1,q-2]_\downarrow}^-_-(\frac12,\frac12)
 $$
 since obviously
 $$
  \textstyle
  ({[\frac32,\frac{c}2]_\downarrow}^-_-{[\frac32,\frac{c}2]_\downarrow}^-_-
 {[1,\frac{(c-1)}2]_\downarrow}^-_-{[1,\frac{(c-1)}2]_\downarrow})_\downarrow
 \leq_s
 [q'+1,q-2]_\downarrow.
 $$
 The  inequality $||\pi ||\leq_s  {[1,q-2]_\downarrow}^-_-(\frac12,\frac12)$ also implies that the second inequality in the proposition holds.

 Consider now the case of even $c$. Then analogously we have
 $$
  \textstyle
 ||\pi||
 \leq_s
  ({[1,\frac{c}2]_\downarrow}^-_-{[1,\frac{c}2]_\downarrow}^-_-
 {[\frac12,\frac{(c-1)}2]_\downarrow}^-_-{[\frac12,\frac{(c-1)}2]_\downarrow}^-_-[1,q']_\downarrow)_\downarrow
$$
 Now from the above inequality we get again
 $$
  \textstyle
 ||\pi||\leq_s
 {[1,q-2]_\downarrow}^-_-(\frac12,\frac12).
 $$
 The above inequality also implies that the second inequality in the proposition holds.

The second possibility for the complementary series is that 
$$
 \textstyle
 \pi= \nu^\alpha \langle[\nu^{-\frac{(c-1)}2}\chi,\nu^{\frac{(c-1)}2}\chi]\rangle \rtimes\pi',
 $$ 
 where $\pi'$ is an irreducible unitarizable unramified representation of $Sp(2q',F)$ (then $c+q'=q$), $c$ is a positive integer, $\chi$ is a self dual character of $F^\times$, $0<\a<1/2$ and $\nu^{\frac12} \langle[\nu^{-\frac{(c-1)}2}\chi,\nu^{\frac{(c-1)}2}\chi]\rangle \rtimes \mathbf 1_{Sp(0,F)}$ reduces. The last reducibility condition and Proposition \ref{prop-red} imply that $c$ is an even number. 
  Now the inductive assumption implies 
 $||\pi||\leq _s
 \big({[1,\frac c2]_\downarrow}^-_-{[\frac12,\frac {c-1}2]_\downarrow}^-_-[1,q']_\downarrow)\big)_\downarrow.
 $ This further implies that $||\pi||\leq _s
 {[1,q-2]_\downarrow}^-_-(1,\frac12)$, which implies the first inequality in the  proposition. Further, 
  $||\pi||\leq_s
 \big({[1,\frac c2]_\downarrow}^-_-{[\frac12,\frac {c-1}2]_\downarrow}^-_-[1,q']_\downarrow)\big)_\downarrow.
 $ implies the second inequality if $c>2$ or if $c=2$ and $q>2$.  For $c=2$ and $q=2$ one directly computes that   
 $||\nu^\alpha \langle[\nu^{-\frac{1}2}\chi,\nu^{\frac{1}2}\chi]\rangle \rtimes \mathbf1_{Sp(0,F)}||=1$. Therefore, the second equity holds also in this case.

 If the complementary series cannot be written in any of the above two forms, then Theorem \ref{theorem-class} implies that it must be a
 complementary series of the form
 $$
\pi= \nu^{x_1}\varphi\times\dots\nu^{x_k}\varphi\rtimes\pi',
 $$
 where $\pi'$ is an irreducible unitarizable unramified representation of $Sp(2q',F)$, $\varphi$ is a self dual character of some $GL(c,F)$ such that $\nu^{1/2} \varphi\rtimes\mathbf1_{Sp(0,F)}$ is irreducible, and $0<x_1,\dots,x_k<1$ satisfy conditions of (3) from Theorem \ref{theorem-class} (we shall use these conditions later).   Observe that $q'+kc=q$ and that $c$ must be odd by Proposition \ref{prop-red} since $\nu^{1/2} \varphi\rtimes\mathbf1_{Sp(0,F)}$ is irreducible.
 We shall consider four  cases.

 First consider the case $c=1$ and $k=1$. 
  Now condition (d) of  Theorem \ref{theorem-class} implies that $\varphi\rtimes\pi_{neg}$ is irreducible. If $\pi'\cong\mathbf 1_{Sp(2(q-1),F)}$, then $\pi'=\pi_{neg}$ and $\chi\t\pi_{neg}$ would be reducible since $q\geq2$ (see Proposition \ref{prop-red}). Thus, $\pi'\not\cong\mathbf 1_{Sp(2(q-1),F)}$
  Suppose $q\geq 3$.
  Now inductive assumption implies that ${[1,q-2]_\downarrow}^-_-(1,\frac12)$ is an upper bound. This implies the first inequality of the proposition. This implies also the second inequality in this case. 
  
  Suppose $q=2$. If $\pi_{neg}$ is a representation of $Sp(2,F)$, then $\pi_{neg}\not\cong\mathbf 1_{Sp(2(q-1),F)}$ as we have observed above. This
 implies that $\pi_{neg}$ is a unitary principal series representation, and then obviously both inequalities of the proposition hold. Suppose now that $\pi_{neg}$ is a representation of $Sp(0,F)$. Then $\pi'$ is a complementary series of $Sp(2,F)$, i.e. $\pi'\cong \nu^\a\mathbf1_{F^\t}\r\mathbf1_{Sp(0,F)}$ with $0<\a<1$. This implies $\varphi={sgn}_{F^\t}$ since $k=1$. But then $\chi\r\pi_{neg}$ is reducible, which is not the case. Thus, this case cannot happen. This completes the proof of the case $c=1$ and $k=1$.

 Consider now the case $c=1$ and $k\geq 2$. Observe that by (c) of the classification Theorem \ref{theorem-class}, at least one $x_i$ is $\leq \frac12$. Thus, ${[1,q']_\downarrow}^-_-(1,\dots,1,\frac12)$ is an upper bound, which obviously implies the first inequality of the proposition. For the second one, observe that the above inequality implies that an upper bound for the left hand side of the second equality is $(q-1)(q-2)/2+3/2$. If $q>2$, then this is obviously $\leq q(q-1)/2$, and therefore the second equity also holds. Suppose that $q=2$. Then (b) of   Theorem \ref{theorem-class} implies that $x_1+x_2\leq 1$. This implies the second inequality in this case.
 
 It remains to prove the proposition  when  $c\geq 3$ (recall, $c$ must be odd for the complementary series that we consider).
 Consider first the case $k=1$. Then we have a following upper bound
 $$
 \textstyle
\big({[1,q']_\downarrow}_-^-
[\frac12,\frac{2c-1}2]_\downarrow
\big)_\downarrow.
 $$
 Recall $q'+c=q$. This obviously implies the first inequality. This inequality implies the second inequality if $q'\geq1$ (since $c\geq2$). It remains to consider the case
 $q'=0$. But then $\varphi\r\mathbf1_{Sp(0,F)}$ is reducible since $c$ is odd (see Proposition \ref{prop-red}), and we cannot have complementary series (by (d) of Theorem \ref{theorem-class}). This completes the proof for this case.

 We are left with the last case $c\geq2$ and $k\geq 2$. Then we must  have at least one exponents between 0 and $\frac12$ by  (c) of Theorem \ref{theorem-class}. Now we have an upper bound as above, except that we have $k$ times $[\frac12,\frac{2c-1}2]_\downarrow$. This obviously implies that the first inequality of the proposition holds. One gets the second inequality from this upper bound since $q'(q'+1)/2+kc^2/2\leq q(q-1)/2$ (one directly gets this using  that $q'+kc=q$ and $k\geq2)$.
\end{proof}


\begin{thebibliography}{99}



\bibitem
{A}
 Arthur, J.,
{\it  The Endoscopic Classification of
Representations: Orthogonal and
Symplectic Groups}, Amer. Math. Soc. Colloq. Publ., vol. 61, Amer. Math. Soc., Providence, RI, 2013.


\bibitem
{BLiSa92}
  Burger, M.,  Li, J.-S.  and  Sarnak, P.,
{\it Ramanujan duals and
automorphic spectrum}, 
 Bull. A.M.S.
 26
no. 2 
 (1992), 253-257.



\bibitem
{Ca}
Casselman, W., {\it Introduction to the theory of admissible representations of p-adic reductive
groups}, preprint (http://www.math.ubc.ca/$\sim$cass/research/pdf/p-adic-book.pdf).

\bibitem
{Ca81}
Casselman, W., {\it  A new nonunitarity argument for $p$-adic representations}, J. Fac. Sci. Univ. Tokyo Sect. IA Math. 28 (1981), no. 3, 907-928 (1982). 


\bibitem
{Cl03}
Clozel, L.,
{\it D\'emonstration de la Conjecture $\tau$},
 Inventiones
Math. 
151 
 (2003),
 297-328.
 



\bibitem
{Cl07}
Clozel, L., 
{\it Spectral Theory of Automorphic forms},
in "Automorphic forms and applications",
IAS/Park City math. ser. 
 12 
 (2007),
   41-93.
   

  
  
   \bibitem
   {ClU04}
  Clozel, L.
  and  Ullmo, E., {\it Equidistribution des points de Hecke}, in ``Contributions to Automorphic Forms, Geometry and Arithmetic",  
  Johns Hopkins University Press, 2004, 193-254

\bibitem
{HJ}
 Hanzer, M. and Jantzen, C. {\it A method of proving non-unitarity of representations of p-adic groups}, J. Lie Theory 22 (2012), no. 4, 1109-1124.

\bibitem
{HT}
 Hanzer, M. and Tadi\'c, M., {it A method of proving non-unitarity of representations of p-adic groups I}, Math. Z. 265 (2010), no. 4, 799-816.

\bibitem
{K67}
 Kazhdan, D., {\em
Connection of the dual space of a group with the structure of its closed
subgroups,}  Functional Anal. Appl. {1} (1967), 63-65.

\bibitem
{LMuT}
 Lapid, E., Mui\'c, G. and Tadi\'c, M., {\em On the generic unitary dual of quasisplit classical groups,}
Int. Math. Res. Not.  no. 26 (2004), 1335--1354.

\bibitem
{LuZu}
 Lubotzky, A. and   Zuk, A.,
{\it  On property $(\tau)$},  preprint. 

\bibitem
{Moe-Ex}
 M\oe glin, C.,
{\it   Sur la classification des s\'eries discr\`etes des groupes
classiques p-adiques: para\-m\`etres de Langlands et exhaustivit\'e},
   J.  Eur. Math. Soc.   4 (2002),
    143-200.

\bibitem
{Moe-fin} 
 M\oe glin, C., {\it Points de r\'eductibilit\'e pour les induites de cuspidales}, J. Algebra 268 (2003),  81-117.
 
 \bibitem
 {Moe-Pac} 
 M\oe glin, C.,
{\it
Classification et Changement de base pour les s\'eries discr\`etes des groupes unitaires $p$-adiques}, Pacific J. Math.   233 (2007),  159-204.

 
\bibitem
{Moe-Howe} 
 M\oe glin, C.,
 {\it
Classification des s\'eries discr\`etes pour certains groupes classiques p-adiques}, in 
"Harmonic analysis, group representations, automorphic forms and invariant theory", Lect. Notes Ser. Inst. Math. Sci. Natl. Univ. Singap. 12, World Sci. Publ., Hackensack, NJ, 2007,  209-245.






\bibitem
{Moe-mult-} 
 M\oe glin, C.,
 {\it Multiplicit\'e 1 dans  les paquets  d'Arthur  aux places p-adiques}, in "On certain $L$-functions", Clay Math. Proc. 13 (2011),  333-374.

\bibitem
{Moe-T}
 M{\oe}glin, C. and Tadi\'c, M.,
{\it  Construction of discrete series for classical $p$-adic groups},
  J. Amer. Math. Soc.
15 (2002),
  715-786.





\bibitem
{MoeW-tran}
 M{\oe}glin, C. and Waldspurger, J.-L.,
 {\it Sur le  transfert  des  traces   tordues   d'un  group  lin\'eaire \`a un  groupe
 classique  p-adique},  Selecta  Mathematica 12 (2006),   433-516.
 



\bibitem
{Mu-no-un}
 Mui\'{c}, G., {\em  On the Non-Unitary Unramified Dual for Classical p-adic Groups,}
Trans. Amer. Math. Soc. { 358} (2006), 4653--4687.



\bibitem
{Mu-unit}
Mui\'{c}, G.,
 {\em On Certain Classes of Unitary
    Representations for Split Classical Groups,}
Canadian J. Math. { 59} (2007), 148-185.




 \bibitem
 {MuT} 
 Mui\'c, G. and   Tadi\'c, M.,
{\it  Unramified 
unitary duals for split classical $p$--adic groups; the
  topology and isolated representations},
in   "On Certain
$L$-functions", Clay Math. Proc. vol. 13, 2011,  375-438.




\bibitem
{Oh02}
Oh, H..
{\it Uniform pointwise bounds for matrix coefficients of unitary representations and applications to Kazhdan constants},
Duke Math. J.,  113 (2002), pp. 133-192.

 \bibitem
 {Sh1}
Shahidi, F.,
{\it  A proof of Langlands conjecture on Plancherel
measures; complementary series for $p$-adic groups},
Ann. of Math.
132
(1990),
 273-330.



\bibitem
{Sh2}
Shahidi, F.,
{\it  Twisted endoscopy and reducibility of induced representations
for 
$p$-adic groups},
Duke Math. J.
66
(1992),
 1-41.
 
 \bibitem
 {Si}
Silberger, A.,
{\it  Special representations of reductive p-adic groups are not
integrable},
Ann. of Math.
111
(1980),
 571-587.


\bibitem
{T-top-dual}
Tadi\'c, M., {\it The topology of the dual space of a reductive group over a $p$-adic field}, Glasnik Mat.
18 (38) (1983), 259-279.



\bibitem
{T86}
Tadi\'c, M.,
{\it Classification of unitary representations in
irreducible representations of general  linear group
(non-archimedean case)},.
 Ann. Sci. \'{E}cole Norm.
Sup.   19 (1986),
 335-382.







\bibitem
{T88}
Tadi\'c, M.,
{\it Geometry of dual spaces of reductive  groups 
(non-archimedean case)},
 J. Analyse Math.
 51 (1988),
 139-181. 
 
 \bibitem{T-ext}
 Tadi\'c, M., {\em
An external approach to unitary representations,}  Bull. Amer. Math. Soc. (N.S.)  { 28} no. 2  (1993), 215--252.
 
 \bibitem
 {T-Str}
Tadi\'c, M.,
{\it  Structure arising from induction and Jacquet modules of
representations of classical
$p$-adic groups},
J. of Algebra 177
(1995),
 1-33.

 
 \bibitem
 {T-CJM} 
Tadi\'c, M.,
{\it On reducibility and
unitarizability for classical $p$-adic groups, some general results}, 
 Canad. J. Math. 61 (2009),  427-450.
 



\bibitem
{T-auto}
Tadi\'c, M.,
{\em
On automorphic duals and isolated representations; new phenomena}, J. Ramanujan Math. Soc. 25 no. 3 (2010),  295-328.


\bibitem
{T13}
Tadi\'c, M.,
{\it On interactions between harmonic analysis and the theory of automorphic forms}, in  Automorphic representations and L-functions, 591Ð650, Tata Inst. Fundam. Res. Stud. Math., 22, Tata Inst. Fund. Res., Mumbai, 2013.


\bibitem
{Ze} 
Zelevinsky, A. V.,
{\it  Induced representations of reductive p-adic groups II. On
irreducible representations of GL(n)},
    Ann. Sci. \'{E}cole Norm. Sup.
  13 (1980),
 165-210.

\end{thebibliography}
\end{document}